\documentclass[letterpaper, 11pt]{amsart}
\usepackage{amssymb,amsmath,amsfonts,amsthm}
\usepackage{graphicx}
\usepackage[normalem]{ulem}
\graphicspath{ {./images/} }
\usepackage{setspace}

\usepackage{mathrsfs}  
\usepackage{psfrag}
\usepackage{mathtools}
\usepackage{color}
\usepackage{todonotes}
\usepackage{enumitem}

\definecolor{Red}{rgb}{1,0,0}
\definecolor{Blue}{rgb}{0,0,1}
\definecolor{Green}{rgb}{0,1,0}
\definecolor{magenta}{rgb}{1,0,.6}
\definecolor{gold}{rgb}{.6,.5,0}
\definecolor{orange}{rgb}{1,0.4,0}
\definecolor{darkgreen1}{rgb}{0, .35, 0}
\definecolor{darkgreen}{rgb}{0, .6, 0}
\definecolor{darkred}{rgb}{.75,0,0}

\usepackage{hyperref}
\hypersetup{colorlinks=true, citecolor=green, linkcolor=blue, hypertexnames=false}

\usepackage{chngcntr}
\counterwithin{equation}{section}

\theoremstyle{plain}

\newtheorem{theorem}{Theorem}[section]
\newtheorem{lemma}[theorem]{Lemma}
\newtheorem{proposition}[theorem]{Proposition}
\newtheorem{corollary}[theorem]{Corollary}

\theoremstyle{remark}
\newtheorem{remark}[theorem]{Remark}
\newtheorem{definition}[theorem]{Definition}

\newcommand\numberthis{\addtocounter{equation}{1}\tag{\theequation}}

\newcommand{\Leb}{\operatorname{vol}}

           \def\ea{\end{array}}
          \def\ec{\end{center}}
     \def\ed{\end{description}}
        \def\ee{\end{equation}}
       \def\eea{\end{eqnarray}}
     \def\eeaa{\end{eqnarray*}}
 \def\et{\end{thebibliography}}

\def\bM{{\bf{M}}}

\newcommand{\interior}[1]{%
	{\kern0pt#1}^{\mathrm{o}}%
}

\def\supp{\operatorname{supp}}

\def\cA{{\mathcal A}}

\def\cC{{\mathcal C}}

\def\cU{{\mathcal U}}

\def\cR{{\mathcal R}}
\def\cB{{\mathcal B}}

\def\cF{{\mathcal F}}

\def\cN{{\mathcal N}}

\def\cR{{\mathcal R}}

\def\vep{\varepsilon}

\def\TT{{\mathbb T}}
\def\RR{{\mathbb R}}

\def\NN{{\mathbb N}}

\def\intr{\operatorname{int}}

	\title[Weak Margulis Measures on Expanding Foliations]{Margulis Measures on Expanding Foliations: Construction and Rigidity}
\date{\today}
\author{J\'er\^ome Buzzi, Yi Shi,  Fan Yang and Jiagang Yang}

\address{}
 \email{}

\thanks{}

\begin{document}

\begin{abstract} 
	Given a diffeomorphism preserving a one-dimensional expanding foliation $\cF$ with homogeneous exponential growth, we construct a family of reference measures on each leaf of the foliation with controlled Jacobian and a Gibbs property. We then prove that for any measure of maximal $u$-entropy, its conditional measures on each leaf must be equivalent to the reference measures. When the measure of maximal $u$-entropy is a Gibbs $\cF$-state (i.e., when the reference measures are equivalent to the leafwise Lebesgue measure), we prove that the log-Jacobian of $f$ must be cohomologous to a constant via a measurable function. We provide  several applications, including the strong and center foliations of Anosov diffeomorphisms, factor over Anosov diffeomorphisms, and perturbations of the time-one map of geodesic flows on surfaces with negative curvature.
\end{abstract}

\maketitle

\tableofcontents

\section{Introduction}
In his celebrated work \cite{Margulis}, Margulis constructed a family of measures $\{\mu^u_x\}$ supported on the unstable foliation of any Anosov diffeomorphism on $\TT^n$ with the following properties:
\begin{enumerate}
	\item (Constant $u$-Jacobian) For each $x\in \bM$, one has 
	$$
	f_*\mu^u_x= e^{-h_{\mathrm{top}}(f)}\mu^u_{f(x)}.
	$$
	\item ($s$-invariance) The stable holonomy $H^s_{x,y}$ maps $\mu^u_x$ to $\mu^u_y$.
\end{enumerate}
Throughout this article, we refer to $\{\mu_x^u\}$ as the (unstable) Margulis measures. The stable Margulis measures $\{\mu^s_x\}$ can be constructed analogously by considering $f^{-1}$. Using the global product structure between the stable and unstable foliations on the universal cover of $\TT^n$, one can then define a measure $\mu = \mu^u\times \mu^s$. The fact that both $\{\mu^u_x\}$ and $\{\mu^s_x\}$ have constant Jacobian $e^{-h_{\mathrm{top}}(f)}$ implies that $\mu$ is $f$-invariant. Margulis further showed that $\mu$ is the unique measure of maximal entropy (MME) of $f$, and used this construction to derive several statistical and topological properties of $f$, most notably the asymptotic counting of periodic orbits (the Prime Orbit Theorem). This development is parallel to the famous work of Bowen \cite{Bowen} on equilibrium states for uniformly hyperbolic diffeomorphisms and provide information that are not clear from Bowen's approach.

More recently, constructions of Margulis-type measures have been carried out for systems beyond uniform hyperbolicity. A notable work is \cite{BFT}, where the authors constructed Margulis measures for perturbations of the time-one map of Anosov flows (which are partially hyperbolic diffeomorphisms), and subsequently established the finiteness of measures of maximal entropy for such systems. In \cite{CPZ} and \cite{CPZ1} the authors uses the Carath\'eodory dimension theory to construct reference measures on the unstable foliation of Anosov systems; when the potential is zero, their construction yields the Margulis measures. A more recent work by Humbert \cite{Humbert} studies Margulis measures on the center-unstable and strong-unstable foliations of Anosov diffeomorphisms on $\TT^3$, and uses them to obtain statistical properties and rigidity for measures of maximal $uu$-entropy. 

Previous constructions of Margulis measures are typically functional analytic in nature, relying on the existence of fixed points of certain operators acting on spaces of leafwise measures. In this paper, we propose a purely geometric construction which does not assume any partially hyperbolic structure. More precisely, we consider diffeomorphisms preserving a foliation $\cF$ on which the dynamics is uniformly expanding. Under a technical condition which we call {\it homogeneous exponential growth (HEG)}, we show that there exists a family of leafwise measures $\{\mu^u_x\}$ satisfying the {\it weak Margulis property}:
$$
\frac{d f^n_*\mu^u_x}{d \mu^u_{f^n(x)}}\in[C^{-1},C]\cdot e^{-nH},
$$
for some uniform constant $C>1$ and $H>0$ depending only on $f$ and $\cF$, and for all $x\in\bM$ and $n\in\NN$. It turns out that $H$ coincides with the topological entropy of $f$ along the foliation $\cF$ (see \cite{HSX} and \cite{HHW}). We also establish a Gibbs-type property relating the measure of a $\cF$-disk to the time it takes for the disk to grow to a fixed size. Precise statements are given in Theorem~\ref{t.properties}. The weak Margulis measures constructed here will serve as the conditional measures of $\cF$-MMEs,  that is, invariant measures whose partial entropy along $\cF$ equals the topological entropy of $\cF$, as illustrated in Theorem~\ref{t.1}. 

We would like to emphasize that the HEG condition, although technical in its formulation, captures a robust geometric property: it requires that the exponential growth of $\cF$-disks under iteration is uniform across the manifold, both in rate and in scale. In particular, it rules out strong inhomogeneities in the expansion along $\cF$. 

This condition is satisfied in a wide range of natural examples. In particular, it holds for co-dimension one Anosov diffeomorphisms (Theorem~\ref{t.Anosov}), derived-from-Anosov diffeomorphisms on $\TT^3$ (Theorem~\ref{t.DA}), partially hyperbolic diffeomorphisms on $3$-nilmanifolds (Theorem~\ref{t.5.4}), as well as for geodesic flows on surfaces of negative curvature and perturbations of their time-one maps (Theorem~\ref{t.geodesic}). 
At present, we are not aware of any example where HEG fails when $\cF$ is minimal.


The weak Margulis measures allows us to obtain rigidity for any expanding foliation.
When the weak Margulis measures are absolutely continuous with respect to the leafwise Lebesgue measure, we  show that the log Jacobian of $f$ along $\cF$ must be cohomologous to a constant through a measurable solution to the cohomologous equation. We would like to remark that this rigidity result does not rely on any (partial) hyperbolicity structure.
In particular, if $f$ is $C^\infty$, $C^1$-close to the time-one map of the geodesic flow, volume preserving, and the Lyapunov exponents coincide with the topological entropy of $f$, then $f$ is the time-one map of a smooth flow whose log Jacobian along both $\cF^s$ and $\cF^u$ are cohomologous to constants. For more applications, see Section \ref{s.5} and \ref{s.geodesic}.


This paper is the first part of a broader project aimed at studying unique ergodicity, transversal invariance, and effective equidistribution for expanding foliations. Central to this program are measures of maximal $\cF$-entropy ($\cF$-MMEs). Just as classical MMEs capture global dynamical complexity and asymptotics, we expect $\cF$-MMEs to reflect the intrinsic geometry and topology of the foliation $\cF$.  Preliminary work along this direction can be found in \cite{UVYY1} and \cite{UVYY2} which link $u$-MMEs of factor Anosov systems to the transversal invariant measures of $\cF^u$, and obtain effective equidistribution for the hitting measure of the unstable foliation on any cross section.

\section{Weak Margulis measures: a general setting}\label{s.construction}

Let $f$ be a $C^1$ diffeomorphism on a compact manifold $\bM$, and $\cF$ an $f$-invariant continuous foliation with   one-dimensional $C^1$ leaves, such that $f|_{\cF}$ is expanding: there exist two constants $1<\lambda < K_f$ such that for every $x\in\bM$ one has 
\begin{equation}\label{e.exp}
1<\lambda < m(Df_x|_{T_x\cF})\le \|Df_x|_{T_x\cF}\|<K_f.
\end{equation}
Here $K_f$ can simply be taken as the $C^1$ norm of $f$.  The goal of this section is to construct a family of measures on each $\cF$-disk with bounded length, which we will call weak Margulis measures. The construction is carried out in Section \ref{ss.construction}, and the properties of these measures are stated in Theorem \ref{t.properties}.

\subsection{Construction}\label{ss.construction}

Given a measure $\mu$ and a subset $A\subset \bM$ with $\mu(A)>0$, we denote by $\mu|_A$ the conditional measure $\nu(\cdot\cap A)/\mu(A)$. In particular, any conditional measure is always a probability measure.

Recall that leaves of $\cF$ are one-dimensional. In particular, all $\cF$-disks are compact curves with positive and finite length. On each leaf of $\cF$ we denote by $\Leb_\cF$ the volume measure induced by the Riemannian structure of $\bM$.
An $\cF$-disk is a one-dimensional disk $D$ (i.e., a $C^1$ embedding from $[0,1]$ to $\bM$) that is contained in a leaf $\cF(x)$ for some $x\in \bM$. For any $\cF$-disk $D$, we denote by $|D|_\cF = \Leb_\cF(D)$ the length of $D$.  We also write $\Leb_D$ for the conditional probability measure $\Leb_{\cF}|_D$. 

\begin{definition}\label{d.HG}
	We say that $\cF$ has homogeneous exponential growth (henceforth homogeneous growth), if there exist $H_\cF>0, C_G>1$ (G for ``growth'') such that every $\cF$-disk $D$ with $|D|_{\cF}\in [1,K_f)$ satisfies 
	\begin{equation}
		C_G^{-1} e^{nH_\cF} \le |f^n(D)|_{\cF}		\le C_Ge^{nH_\cF},\,\, \forall n\in\NN.
	\end{equation}	
\end{definition}
Note that for every $\cF$-disk $D$ with $|D|_\cF\le 1$, there must exist some $n\in\NN$ such that $|f^{n}(D)|_\cF\in [1,K_f)$; here $n=0$ if $|D|_\cF = 1$. We thus define 
\begin{equation}\label{e.n}
	n_0(D) = \min\{k\in\NN: |f^k(D)|_\cF \ge 1\}.
\end{equation}

\begin{remark}\label{r.vep}{
	Here the requirement that $|D|_\cF\in [1,K_f)$ is only for the sake of convenience, and all the results below still hold with $[1,K_f)$ replaced by $[\vep,K_f\vep)$ where $\vep>0$ is an arbitrary constant.	In this case, the definition of $n_0$ should be modified as 
	$$
		n_0(D) = \min\{k\in\NN: |f^k(D)|_\cF \ge \vep\},
	$$
	and all the results in this paper remain valid.}
\end{remark}

Below we briefly recall two different definitions of the topological entropy of $f$ along an expanding foliation $\cF$, which is defined similarly as the unstable topological entropy in the literature. These definitions do not require $\cF$ to have homogeneous growth.

Given an expanding foliation of a $C^1$ diffeomorphism $f$, \cite{HSX} defines the $\cF$-topological entropy as the volume growth rate of $f$ on $\cF$:
$$
\chi_\cF(x,r) = \limsup_{n\to\infty} \frac1n \log\left(\Leb_\cF \left(f^n(B_r^\cF(x))\right) \right),
$$
where $B^\cF_r(x)$ is the $\cF$-disk on $\cF(x)$ centered at $x$ with radius $r>0$. Then, define 
$$
\chi_\cF(f) = \sup_{x\in\bM} \chi_\cF(x,r)
$$
as the maximum volume growth rate of $\cF$ under $f$, which can be shown to be independent of $r$.

On the other hand, \cite[Definition 1.4]{HHW} defines the $\cF$-topological entropy $h^\cF_{top}(f)$ as the exponential growth rate of leafwise separated sets; the precise definition is omitted in this paper. By \cite[Theorem C]{HHW} one has 
$
\chi_\cF(f) = h^\cF_{top}(f).
$

 The following lemma follows immediately from the definition, and thus the proof is omitted.
	\begin{lemma}
		If $\cF$ has homogeneous growth, then one must have 
		$$
		H_\cF =  h^\cF_{top}(f)= \chi_\cF(f)  =  \chi_\cF(x,r), \,\,\forall x\in\bM, r>0. 
		$$
	\end{lemma}

Now we construct the weak Margulis measures on $\cF$.
\begin{definition}\label{d.G}
	Assume that $\cF$ is an $f$-invariant expanding foliation with homogeneous growth. For any $\cF$-disk $D$, define
	$$
	\Gamma(D) = \left\{\nu: \nu \mbox{ is a limit point of } (f^{-n})_{*}\left(\Leb_\cF|_{f^n(D)}\right)\right\} \subset \operatorname{Prob}(\overline D)
	$$
	where $\Leb_\cF|_{f^n(D)}$ is the leafwise Lebesgue measure $\Leb_\cF$ conditioned on $f^n(D)$ (so that it becomes a probability measure), the limit is taken in the weak-* topology on $\bM$, and $\operatorname{Prob}(\overline D)$ denotes the collection of probability measures on $\overline D$.
\end{definition}

Note that $\Gamma(D) \ne\emptyset$ for every such $D$, since the space of probability measures on $\bM$ is compact.  Even though the definition applies to $\cF$-disks with arbitrary length, below we will only consider the case when $|D|_\cF\le 1.$

\subsection{Properties of weak Margulis measures}

From now on, for $a,b>0$ and $C\ge 1$ we write $a\asymp_C b$ if $C^{-1}b<a<Cb$. Note that
$$
a\asymp_{C_1}b \mbox{ and } b\asymp_{C_2}c \implies a \asymp_{C_1C_2} c.
$$
Following this notation, $\cF$ has homogeneous growth if one has
$$
|f^{n}(D)|_\cF\asymp_{C_G} e^{nH_\cF}
$$
for every $\cF$-disk $D$ with $|D|_\cF\in [1,K_f)$ and $n\in\NN$.

The properties of measures in $\Gamma$ are summarized by the following theorem:
\begin{theorem}\label{t.properties}
	Assume that $\cF$ is an expanding foliation for the $C^1$ diffeomorphism $f$ with homogeneous exponential growth. Then, there exists $L>1$ such that for any $\cF$-disk $D$, the measures in $\Gamma(D)$ given by Definition \ref{d.G} have the following properties:
	\begin{enumerate}
		\item (Invariance) $f_*(\Gamma(D)) = \Gamma (f(D))$. 
		\item (Gibbs property) Assume that $|D|_\cF\le 1$. Then, for any sub $\cF$-disk $I\subset D$ and any $\nu\in\Gamma(D)$, we have
		\begin{equation*}\label{e.Gibbs}
			\nu(I)\asymp_L \frac{e^{-n_0(I)H_\cF}}{e^{-n_0(D)H_\cF}}.
		\end{equation*}
		Also, if $|D|_\cF = 1$ then $n_0(D) =0$ and so $\nu(I)\asymp_L {e^{-n_0(I)H_\cF}}.$ Consequently, $\nu$ has no atom and $\supp(\nu) = D.$
		\item (Uniformly equivalent at intersections) If $D = \intr(D_1)\cap \intr(D_2)\ne\emptyset$, then any $\nu_1\in\Gamma(D_1)$ and $\nu_2\in\Gamma(D_2)$ are equivalent on $D$, with 
		$$
		\frac{d\nu_1|_{D}}{d\nu_2|_D}\asymp_{L^4}1.
		$$
		\item (Conformal) For any $\cF$-disk $D_1$ with $|D_1|_{\cF}\le 1$ and $n\in\NN$, let $D_2\subset f^n(D_1)$ be any $\cF$-disk with $|D_2|_\cF\le 1$. Then, one has 
		$$
		\frac{d (f^n)_*\nu_1}{d\nu_2} \asymp_{L^2}e^{-nH_\cF} \cdot \frac{e^{-n_0(D_2)H_\cF}}{e^{-n_0(D_1)H_\cF}},\,\, \forall \nu_1\in\Gamma (D_1),\nu_2\in \Gamma(D_2).
		$$ 
		In particular, if $|D_1|_\cF = |D_2|_\cF = 1$, then $\frac{d (f^n)_*\nu_1}{d\nu_2} \asymp_{L^2} e^{-nH_\cF}$.
	\end{enumerate}
\end{theorem}

\begin{remark}
	The denominator $e^{-n_0(D)H_\cF}$ in the Gibbs property, as well as the factor $\frac{e^{-n_0(D_2)H_\cF}}{e^{-n_0(D_1)H_\cF}}$ in the conformal property, should be considered as a normalizing factor to compensate for the various sizes of $\cF$-disks.
\end{remark}

\begin{remark}
	A key feature of our approach is that we avoid considering the measurability of $D\mapsto\Gamma(D)$ or of $f$-invariant sections $\nu_x\in\Gamma(D_x)$ for some choice of $D_x\subset\cF(x)$. Instead, the measurability of $e^{-n_0(I) H_\cF}$ will play an important role.
\end{remark}
The rest of this section is devoted to the proof of Theorem \ref{t.properties}, which is split into several lemmas.

\begin{lemma}\label{l.inv}
We have 	$f_*(\Gamma(D)) = \Gamma (f(D))$. 
\end{lemma}
\begin{proof}
	$\nu\in \Gamma(D)$ if and only if there exists a sequence of positive integers $n_1<n_2\cdots$ such that $\nu = \lim_{k\to\infty} (f^{-n_k})_*(\Leb_\cF|_{f^{n_k}(D)}).$ This is equivalent to 
	$$
	f_*(\nu) = \lim_{k\to\infty} (f^{-(n_k-1)})_*(\Leb_\cF|_{f^{n_k-1}(f(D))}),
	$$
	which is true if and only if $f_*(\nu)\in\Gamma(f(D))$.
\end{proof}

\begin{lemma}\label{p.gibbs}
	There exists $L>1$ such that for any $\cF$-disk $D$ with $|D|_\cF\le 1$, any sub $\cF$-disk $I\subset D$ and any $n \ge n_0(I)$, it holds that  
	$$
	(f^{-n})_{*}\left(\Leb_\cF|_{f^n(D)}\right)(I) \asymp_L \frac{e^{-n_0(I)H_\cF}}{e^{-n_0(D)H_\cF}}.
	$$
	More precisely, it is possible to take $L = (C_G)^2$.
\end{lemma}

\begin{proof} Recall that $\cF$-disks are compact curves with positive and finite length. We write
	\begin{align*}
		(f^{-n})_{*}\left(\Leb_\cF|_{f^n(D)}\right)(I) &= \frac{|f^n(I)|_\cF}{|f^n(D)|_\cF}.  
	\end{align*}
	For the numerator, note that $|f^{n_0(I)}(I)|_\cF\in[1,K_f)$, and by Definition \ref{d.HG} we obtain
	$$
	|f^n(I)|_\cF \asymp_{C_G} e^{(n-n_0(D))H_\cF}.
	$$ 
	For the denominator, the same argument yields 
	$$
	|f^n(D)|_\cF \asymp_{C_G} e^{(n-n_0(D))H_\cF}.
	$$
	that is, 
	$$
	|f^n(D)|_\cF^{-1} \asymp_{C_G} e^{-(n-n_0(I))H_\cF}.
	$$
	Combining the two, we conclude that 
	$$
	(f^{-n})_{*}\left(\Leb_\cF|_{f^n(D)}\right)(I) \asymp_{(C_G)^2} \frac{e^{-n_0(I)H_\cF}}{e^{-n_0(D)H_\cF}}, 
	$$
	as required.
\end{proof}
{Note that the above estimate in Lemma 2.7 passes to the weak-* limit $\nu$. Indeed, it is uniform in $n\ge n_0(I)$ and the disk $I$ can always be approximated by some other $\cF$-disks $J,K$ satisfying $J\subset  \operatorname{int}(I)\subset I \subset \operatorname{int}(K)$ with $n_0(J),n_0(K) = n_0(I) \pm1$.} Increasing $L$ if necessary, we obtain the following corollary of Proposition \ref{p.gibbs}.
\begin{lemma}\label{p.ac1}
	For any $\cF$-disks $I\subset D$ with $|D|_\cF\le 1$ and any $\nu\in \Gamma(D)$, it holds that  
	$$
	\nu(I)\asymp_{L} \frac{e^{-n_0(I)H_\cF}}{e^{-n_0(D)H_\cF}}.
	$$
	Consequently, any $\nu,\nu'\in\Gamma(D)$ are equivalent to each other, with 
	$$
	\frac{d\nu}{d\nu'}\asymp_{L^2} 1.
	$$
\end{lemma}

Next, we improve Lemma \ref{p.ac1} by considering the intersection of two $\cF$-disks.

\begin{lemma}\label{p.ac2}
	Let $D_1,D_2$ be two $\cF$-disks with $|D_i|_\cF\le 1$, and take any $\nu_i\in\Gamma(D_i), i = 1,2$. Write $D:= \intr(D_1)\cap \intr(D_2)$. Then, the measures $\nu_1|_D$ and $\nu_2|_D$ are equivalent to each other, with 
	$$
	\frac{d\nu_1|_D}{d\nu_2|_D} \asymp_{L^4} 1.
	$$ 
\end{lemma} 

\begin{proof}
	Let $I\subset D$ be any $\cF$-disk. Then, by Proposition \ref{p.ac1} we obtain 
	$$
	\nu_i(I)\asymp_{L} \frac{e^{-n_0(I)H_\cF}}{e^{-n_0(D_i)H_\cF}}, i=1,2.
	$$
	Note that the same estimate holds when $I = D.$ This leads to 
	\begin{align*}
		\frac{\nu_1|_D(I)}{\nu_2|_D(D)} &= \frac{\nu_1(I)}{\nu_2(I)}\cdot\frac{\nu_2(D)}{\nu_1(D)}\\
		&\asymp_{L^4} \frac{\frac{e^{-n_0(I)H_\cF}}{e^{-n_0(D_1)H_\cF}}}{\frac{e^{-n_0(I)H_\cF}}{e^{-n_0(D_2)H_\cF}}} \cdot \frac{\frac{e^{-n_0(D)H_\cF}}{e^{-n_0(D_2)H_\cF}}}{\frac{e^{-n_0(D)H_\cF}}{e^{-n_0(D_1)H_\cF}}}\\
		&\asymp_{L^4} 1.
	\end{align*}
	Since $I$ is arbitrary, we conclude that $\nu_1|_D$, $\nu_2|_D$ are equivalent to each other, with Radon-Nikodym derivative between $L^{-4}$ and $L^4.$
\end{proof}

Finally, we consider the conformal property of measures in $\Gamma(D)$. 
\begin{lemma}\label{l.jac}
	For any $\cF$-disk $D_1$ with $|D_1|_{\cF}\le 1$ and $n\in\NN$, let $D_2\subset f^n(D_1)$ be any $\cF$-disk with $|D_2|_\cF\le 1$. Then, one has 
	$$
	\frac{d (f^n)_*\nu_1}{d\nu_2} \asymp_{L^2} e^{-nH_\cF}\cdot \frac{e^{-n_0(D_2)H_\cF}}{e^{-n_0(D_1)H_\cF}},\,\, \forall \nu_1\in\Gamma (D_1),\nu_2\in \Gamma(D_2).
	$$ 
\end{lemma}
\begin{proof}
	Let $I\subset D_2$ be any $\cF$-disk. Let $\nu_1\in D_1,\nu_2\in D_2$ be arbitrary. By Lemma \ref{p.ac1} we have
	\begin{equation}\label{e.c1}
		\nu_2(I)\asymp_L\frac{e^{-n_0(I)H_\cF}}{e^{-n_0(D_2)H_\cF}}.
	\end{equation}
	On the other hand, $f^{-n}(I)\subset f^{-n}(D_2)\subset D_1$ is also an $\cF$-disk. Consequently, Lemma \ref{p.ac1} gives 
	\begin{equation*}
		\nu_1(f^{-n}(I))\asymp_L\frac{e^{-n_0(f^{-n}(I))H_\cF}}{e^{-n_0(D_1)H_\cF}}.
	\end{equation*}
	Note that $n_0(f^{-n}(I)) = n + n_0(I)$, so the above equality becomes 
	\begin{equation}\label{e.c2}
		(f^n)_*\nu_1(I)\asymp_L\frac{e^{-(n+n_0((I)))H_\cF}}{e^{-n_0(D_1)H_\cF}}.
	\end{equation}
	Combining \eqref{e.c1} and \eqref{e.c2}, we conclude
	\begin{align*}
		\frac{(f^n)_*\nu_1(I)}{\nu_2(I)}&\asymp_{L^2}\frac{e^{-(n+n_0((I)))H_\cF}}{e^{-n_0(D_1)H_\cF}} \cdot \frac{e^{-n_0(D_2)H_\cF}}{e^{-n_0(I)H_\cF}} \\
		&\asymp_{L^2} e^{-nH_\cF} \cdot \frac{e^{-n_0(D_2)H_\cF}}{e^{-n_0(D_1)H_\cF}}.
	\end{align*}
	Now the lemma follows by noting that the right-hand side does not depend on the choice of $I$.
\end{proof}

\begin{proof}[Proof of Theorem \ref{t.properties}]
	The invariance follows from Lemma \ref{l.inv}. The Gibbs property is Lemma \ref{p.ac1}. The uniform equivalences of measures on the intersection is Lemma \ref{p.ac2}. Finally, the conformal property is proven in Lemma \ref{l.jac}. We conclude the proof of Theorem \ref{t.properties}.
\end{proof}

\section{Weak Margulis measures and measures of maximal $\cF$-entropy}\label{s.3}
\subsection{Setup}\label{ss.partition}
We let $\mu$ be a measure of maximal $\cF$-entropy, whose existence is given by \cite[Theorem A]{Y16} and \cite[Proposition 2.15]{HHW}. From the discussion in Section \ref{ss.construction}, we have 
$$
h^\cF_{\mu}(f) = h^\cF_{top}(f) = H_\cF.
$$

Next we consider the conditional measures of $\mu$. For definiteness, we will fix a measurable partition which we will describe below.
We let $\cB_{k}, k=1,\cdots, N_0$ be a family of foliation boxes that cover $\bM$. By cutting each $\cB_{k}$ into three pieces whenever $\cB_{k}\cap \cB_{j}\ne\emptyset$, we obtain a finite measurable partition $\cA = \{A_k:k=1,\cdots, N_1\}$ with the property that $\mu(\partial A_k) = 0$ for every $k$. {We will also assume that $\cA$ satisfies the small boundary condition as in \cite{Y16}.} Elements of $\cA$ have a form of product structure: there exists $a>0$ such that for every $x\in \bM$, denoting by $\cF_\cA(x)$ the connected component of $\cF(x)\cap \cA(x)$ that contains $x$, one has 
\begin{equation}\label{e.size}
	a\le|\cF_\cA(x)|_\cF \le 1, \forall x\in\bM.
\end{equation} 

In each $A_k$, the collection $\{\cF_\cA(x): x\in A_k\}$ form a measurable partition of $A_k$. This gives a measurable partition of $\bM$, which we denote by $\cF_\cA$. We then obtain a family of conditional measures of $\mu$, which we denote by $\mu^{\cF,\cA}_x$. On the other hand, each $\cF_\cA(x)$ is an $\cF$-disk with length at most one, and therefore we can define a family of reference measures on $\cF_\cA(x)$, denoted by $\Gamma(\cF_\cA(x))$, following Definition \ref{d.G}. 

This definition only depends on the disk $\cF_\cA(x)$ and not on the point in the sense that if $y\in \cF_\cA(x)$ then $\Gamma(\cF_\cA(y)) = \Gamma(\cF_\cA(x))$. Note that each measure in $\Gamma(\cF_\cA(x))$ is a probability measure with full support and no atom, thanks to Theorem \ref{t.properties}.  We have the following proposition which is the counterpart of Proposition \ref{p.ac1} on $\cF_\cA(x)$. Since the length of $\cF_\cA(x)$ is bounded away from zero, by Theorem \ref{t.properties} (2) we obtain:

%

\begin{proposition}\label{p.ac4}
	There exists a constant $\widetilde L$ such that for any $x\in\bM$, any $\cF$-disk $I\subset \cF_\cA(x)$ and any $\nu\in \Gamma(\cF_\cA(x))$, we have  
	$$
	\nu(I)\asymp_{\widetilde L} e^{-n_0(I) H_\cF}.
	$$
\end{proposition}

\begin{proof}
	This is a simple consequence of the Gibbs property (Theorem \ref{t.properties} (2)) and \eqref{e.size}, which indicates that $0\le \sup_xn_0(\cF_\cA(x))\le N$ for some constant $N\in\NN$. One can take $\widetilde L = L\cdot e^{NH_\cF}$ where $L>1$ is the constant given by Theorem \ref{t.properties}.
\end{proof}

The main result of this section is the following theorem. 

\begin{theorem}\label{t.1}
	Let $\cF$ be an expanding foliation for a $C^{1+\alpha}$ diffeomorphism $f$ with homogeneous growth, and $\mu$ be any ergodic measure of maximal $\cF$-entropy. Then, for $\mu$-almost every point $x$, $\mu^{\cF,\cA}_x$ and any $\nu\in \Gamma(\cF_\cA(x))$ are equivalent to each other, with Radon-Nikodym derivative within $[L_1^{-1}, L_1]$ for some constant $L_1>1$ independent of $x$. {In particular, the support of every ergodic measure of maximal $\cF$-entropy consists of entire $\cF$-leaves.}
\end{theorem}

\begin{remark}
	While the construction and Theorem \ref{t.properties} require $f$ to be $C^1$, here the $C^{1+\alpha}$ regularity is needed. This is because of Lemma \ref{l.b1} which requires a distortion control.
\end{remark}

The proof of Theorem \ref{t.1} occupies the rest of this section. 

\subsection{The proof of $\mu_x^{\cF,\cA}\prec \nu$}\label{ss.3.2}
First we show that  $\mu_x^{\cF,\cA}\prec \nu$. This will be needed later when we establish the other direction of the absolute continuity, and obtain the uniform estimate for the Radon-Nikodym derivative; see Proposition \ref{p.2}.
	
Henceforth, to simplify notation we will write $\mu_x = \mu_x^{\cF,\cA}$ since $\cF,\cA$ have been fixed throughout.	We start with the following lemma regarding the monotonicity of partial $\cF$-entropy. For any invariant probability measure $\xi$, we denote by $h^\cF_\xi(f)$ its {\em partial $\cF$-entropy}; for further discussions, see \cite{Y16} and \cite{HHW}.
\begin{lemma}\label{l.mono}
	Let $\xi$ be any ergodic invariant probability. Then the sequence
	$$
	\frac1n H_\xi\left(\bigvee_{j=1}^n f^{-j}(\cF_\cA)\Big|\cF_\cA\right) 
	$$ 
	decreases monotonically to $h_\xi^\cF(f)$.
\end{lemma}	
\begin{proof}
	The idea is taken from \cite[Proposition 4.1]{Y16}. We write
	\begin{align*}
		&H_\xi\left(\bigvee_{j=1}^n f^{-j}(\cF_\cA)\Big|\cF_\cA\right)\\ &= H_\xi\left(f^{-1}(\cF_\cA)|\cF_\cA\right) + H_\xi\left(f^{-2}\cF_\cA|\cF_\cA\vee f^{-1}(\cF_\cA)\right)+\cdots\\
		&\hspace{1cm} + H_\xi\left(f^{-n}(\cF_\cA)\Big|\bigvee_{j=0}^{n-1}f^{-j}(\cF_\cA)\right)\\
		&=H_\xi\left(\cF_\cA|f(\cF_\cA)\right) + H_\xi\left(\cF_\cA|f(\cF_\cA)\vee f^2(\cF_\cA)\right)+\cdots\\
		&\hspace{1cm} + H_\xi\left(\cF_\cA\Big|\bigvee_{j=1}^{n}f^{j}(\cF_\cA)\right).
	\end{align*}
	We note that $\bigvee_{j=1}^{n}f^{j}(\cF_\cA)$ is an increasing sequence of partitions. Furthermore, $H_\xi(\cF_\cA|f(\cF_\cA))<\infty$. By the Martingale Convergence Theorem (alternatively, see \cite[Section 5]{Rokhlin}), we get that 
	$$
	H_\xi\left(\cF_\cA\Big|\bigvee_{j=1}^{n}f^{j}(\cF_\cA)\right)\searrow  H_\xi\left(\cF_\cA\Big|\bigvee_{j=1}^{\infty}f^{j}(\cF_\cA)\right)=h_\xi^\cF(f).
	$$
	It then follows that $\frac1n H_\xi\left(\bigvee_{j=1}^n f^{-j}(\cF_\cA)\Big|\cF_\cA\right)$ also monotonically decreases to $h_\xi^\cF(f)$. 
\end{proof}

Now let $\mu$ be an ergodic measure of maximal $\cF$-entropy. We need the following {\em a priori} bound on $H_\mu\left(\bigvee_{j=1}^n f^{-j}(\cF_\cA)\Big| \cF_\cA\right)$, {which satisfies 
\begin{align*}
	H_\mu\left(\bigvee_{j=1}^n f^{-j}(\cF_\cA)\Big|\cF_\cA\right) &= 	\int H_{\mu}\left(\bigvee_{j=1}^n f^{-j}(\cF_\cA)\Big|\cF_\cA\right)\,d\mu(x)\\
	&=	\int H_{\mu_x}\left(\bigvee_{j=1}^n f^{-j}(\cF_\cA)\right)\,d\mu(x).
\end{align*}
}

\begin{lemma}\label{l.apbd}
	There exists a constant $L_2>0$ such that for any $n\in\NN$ and $\mu$-almost every $x\in \bM$, one has 
	$$
	H_{\mu_x}\left(\bigvee_{j=1}^n f^{-j}(\cF_\cA)\right) \le nH_\cF + L_2.
	$$
\end{lemma}
 
For each $j= 1,\cdots, n$ we write $(f^{-j}(\cF_\cA))_{\cF_\cA(x)}$ for the collection of those elements of $f^{-j}(\cF_\cA)$ that intersect with $\cF_\cA(x)$. Similarly, we write $(\cF_\cA^n)_{\cF_\cA(x)}$ for the elements of  $\bigvee_{j=1}^n f^{-j}(\cF_\cA)$ that intersects $\cF_\cA(x)$. The proof of the previous lemma requires the following relation between cardinalities of $(f^{-j}(\cF_\cA))_{\cF_\cA(x)}$  and that of $(\cF_\cA^n)_{\cF_\cA(x)}$.
\begin{lemma}\label{l.count}
	For every $n\in\NN$ it holds that
	$$
	\#(\cF_\cA^n)_{\cF_\cA(x)}\le 2 \sum_{j=1}^{n} \#(f^{-j}(\cF_\cA))_{\cF_\cA(x)}.
	$$
\end{lemma}
\begin{proof}
	{
	In restricting to some $\cF_A(x)$, the partition $\cF^k_\cA = \bigvee_{j=1}^k f^{-j}\cF_\cA$ is obtained by iterating elements of $\cF_\cA$ (which are one-dimensional disks) with the diffeomorphism $f^{-k}$ and  cut it into two pieces whenever it meets an endpoint of some element of $(f^{-j}(\cF_\cA))_{\cF_\cA(x)}$ (which are also one-dimensional disks with two endpoints) for some $j<k$. In this way, each $A\in (f^{-k}(\cF_\cA))_{\cF_\cA(x)}$ is cut into $k(A) + 1$ pieces when $k(A)$ is the total number of endpoints that meet $A$. Furthermore, endpoints counted by $k(A)$ will not meet any other $A'\in (f^{-k}(\cF_\cA))_{\cF_\cA(x)}$. This shows that 
	\begin{align*}
		\#(\cF_\cA^k)_{\cF_\cA(x)}&\le \sum_{A\in (f^{-k}(\cF_\cA))_{\cF_\cA(x)}}(k(A) + 1)\\
		& = \sum_{A\in (f^{-k}(\cF_\cA))_{\cF_\cA(x)}}k(A) + \# (f^{-k}(\cF_\cA))_{\cF_\cA(x)}\\
		&\le  2\sum_{j=1}^{k-1} \#(f^{-j}(\cF_\cA))_{\cF_\cA(x)} +  \# (f^{-k}(\cF_\cA))_{\cF_\cA(x)}\\
		&\le 2 \sum_{j=1}^{k} \#(f^{-j}(\cF_\cA))_{\cF_\cA(x)}.
	\end{align*}	
	
 }
	
%
\end{proof}

Now we are ready to prove Lemma \ref{l.apbd}.
\begin{proof}[Proof of Lemma \ref{l.apbd}]
	Recall that if $\cC$ is a partition, and $\xi$ is a probability measure (not necessarily invariant) for which only finitely many elements of $\cC$ have positive measure, then it holds that 
	$$
	H_\xi(\cC)\le \log\#\cC_\xi
	$$ 
	where $\cC_\xi$ are those elements of $\cC$ whose $\xi$-measures are positive.
	
	We apply this observation with $\xi = \mu_x$ and $\cC_\xi = (\cF_\cA^n)_{\cF_\cA(x)}$, whose cardinality is estimated using Lemma \ref{l.count}:
	\begin{align*}
		&\#(\cF_\cA^n)_{\cF_\cA(x)}\\
		&\le  2 \sum_{j=1}^{n} \#(f^{-j}(\cF_\cA))_{\cF_\cA(x)}.\\
		&\le 2\sum_{j=1}^{n}  \#\left\{\cF_\cA(z): z\in\bM,\cF_\cA(z)\cap f^j(\cF_\cA(x))\ne\emptyset \right\}\\
	\numberthis\label{e.apbd1}	&= 2\left(\sum_{j=1}^{n_0(\cF_\cA(x))-1} + \sum_{j=n_0(\cF_\cA(x))}^{n}\right)\#\left\{\cF_\cA(z): \cF_\cA(z)\subset f^j(\cF_\cA(x)) \right\} + 4.
	\end{align*}
	Here the additional term of $4$ comes from the two endpoints of $f^j(\cF_\cA(x))$ where $f^j(\cF_\cA(x))$ intersects with some $\cF_\cA(z)$ but does not contain them.
	
	\eqref{e.size} implies that $n_0(\cF_\cA(x))$ is bounded by a constant independent of $x$. This allows us to bound the first summation by some constant $c_1>0$ that only depends on $f$ and $\cF_\cA$ but not on $x$ or $n$. For the second summation, note that for each $j\ge n_0(\cF_\cA(x))$, by the homogeneous growth of $\cF$, one has 
	$$
	|f^j(\cF_\cA(x))|_\cF \asymp_{C_G} e^{(j-n_0(\cF_\cA(x)))H_\cF}.
	$$
	On the other hand, each $\cF_\cA(z) \subset f^j(\cF_\cA(x))$ is a $\cF$-disk with length at least $a>0$ due to \eqref{e.size}. We see that 
	\begin{equation}\label{e.a0}
		\#\left\{\cF_\cA(z): \cF_\cA(z)\subset f^j(\cF_\cA(x)) \right\}\le c_2 \frac{e^{(j-n_0(\cF_\cA(x)))H_\cF}}{a}.
	\end{equation}
	Using again the fact that $n_0(\cF_\cA(x))$ is bounded, we obtain from \eqref{e.apbd1} and \eqref{e.a0} and 
	\begin{equation}\label{e.a1}
		\#(\cF_\cA^n)_{\cF_\cA(x)}\le c_1 + c_3 e^{nH_\cF} \le c_4e^{nH_\cF} \,\, \mbox{ for every $n\in\NN$},
	\end{equation}
	where all $c_i$ are positive constants that depend on $\cF$ and $\cA$ but not on $n$ or $x$. Then Lemma \ref{l.apbd} follows by taking logarithm.
	
\end{proof}


We continue our proof of $\mu_x\prec\nu$. Let $\ell>0$; we defined the following ``bad'' set:
	$$
	\cB_\ell = \left\{x\in\bM: \mu_x \mbox{ is well-defined, and }\limsup_{r\to 0} \frac{\mu_x(B^\cF_r(x))}{e^{-n_0(B_r^\cF(x))H_\cF}}	\ge \ell\right\},
	$$
	where $B_r^\cF(x)$ denotes the $\cF$-disk centered at $x$ with radius $r$. 
	
	Note that the numerator, $\mu_x(B^\cF_r(x))$, is a $\mu$-measurable function of $(r,x)$. Meanwhile, $n_0(B_r^\cF(x))$ is a locally constant function of $(r,x)$ whenever $|f^{n_0(B_r^\cF(x))}(n_0(B_r^\cF(x)))|_\cF\ne 1$ and is therefore Borel measurable. It follows that the function $h(r,x):={\mu_x(B^\cF_r(x))}/{e^{-n_0(B_r^\cF(x))H_\cF}}$ is $\mu$-measurable w.r.t.\ $(r,x)$. Let $\pi_x:\RR\times\bM\to\bM$ be the projection to the second coordinate, and  $F_n(x) = \sup_{0<r<1/n} h(r,x)$. Then for each $n\in\NN$ and $a\in\RR$, 
		$$
		\{F_{n}(x)>a\} = \pi_x\left\{(r,x): 0<r<1/n,  h(r,x)>a\right\}
		$$ 
	is the projection of a Borel set in the Polish space $\RR\times \bM$, which must be analytic and therefore universally measurable (see, for instance, \cite{H}). In particular, $\lim_n F_{n}(x)$ is  universally measurable, and so $\cB_\ell = \{x:\lim_n F_n(x)\ge \ell\}$ is universally measurable for each $\ell>0$. Furthermore, $\{\cB_\ell\}_\ell$ is nested in the sense that $\cB_\ell\supset \cB_{\ell'} $ whenever $\ell' > \ell.$ Denote by $\cB_\infty = \bigcap_{\ell\in\NN} \cB_\ell$ and note that it is also universally measurable, and in particular, measurable w.r.t. $\mu$.

\begin{lemma}\label{l.b1}
	$\cB_\infty$ is $f$-invariant, and consequently $\mu(\cB_\infty) = 0$ or $1$.
\end{lemma}	
\begin{proof}
	Assume that $x\in \cB_\infty$; there exists $r_n\to 0$ such that 
	$$
	\frac{\mu_x(B^\cF_{r_n}(x))}{e^{-n_0(B_{r_n}^\cF(x))H_\cF}}\to\infty \mbox{ as }n\to\infty.
	$$
	
 	By the uniqueness of the disintegration,  for almost every $y\in \cF_\cA(x)\cap f^{-1}(\cF_\cA(f(x)))$, one has
	$$
	\frac{df_*(\mu_x)}{d \mu_{f(x)}}(y) = c_1
	$$
	where $c_1$ is a constant depending on $\cF_\cA(x)\cap f^{-1}(\cF_\cA(f(x)))$. In particular, this holds for every $y\in B_r^\cF(x)$ provided that $r$ is sufficiently small. We use this to obtain that
	\begin{equation}\label{e.341}
		\frac{\mu_{f(x)}(f(B^\cF_{r_n}(x)))}{e^{-n_0(f(B_{r_n}^\cF(x)))H_\cF}}\to\infty
	\end{equation}
	where we use the fact that $n_0(f(B_{r_n}^\cF(x))) = n_0(B_{r_n}^\cF(x))-1$.
	
	On the other hand, by \eqref{e.exp} we see that 
	\begin{equation}\label{e.342}
	B_{\lambda \cdot r_n}^\cF (f(x))	\subset f(B^\cF_{r_n}(x)) \subset B_{K_f\cdot r_n}^\cF (f(x))
	\end{equation}
	for any $r>0$. 
	Furthermore, since $f$ is $C^{1+\alpha}$, a standard distortion argument gives 
	\begin{equation}\label{e.dist}
		|n_0(B^\cF_{\lambda r}(x)) - n_0(B^\cF_{K_f r}(x))| < c_2
	\end{equation} 
 	for some constant $c_2>0$ independent of $x$ and $r$ (as long as it is small).
 	Combining \eqref{e.342}, \eqref{e.dist} with \eqref{e.341},  we conclude that
 	$$
 	\frac{\mu_{f(x)}(B^\cF_{K_f r_n}(f(x)))}{e^{-n_0(B_{K_f r_n}^\cF(f(x)))H_\cF}}\to\infty,
 	$$
 	and therefore $f(x)\in \cB_\infty$. This finishes the proof of Lemma \ref{l.b1}.
\end{proof}

The rest of the proof is dedicated to show $\mu(\cB_\infty) = 0$. We will assume that this is not the case, and aim to obtain a contradiction with Lemma \ref{l.mono}.
\begin{lemma}[The Key Lemma]\label{l.B}
If $\mu(\cB_\infty) = 1$, then for $\mu$-almost every $x\in\bM$ and any $b>0$, there exists $N= N(x,b)\in\NN$ such that for every $n\ge N(x,b)$, one has 
$$
H_{\mu_x}\left(\bigvee_{j=1}^n f^{-j}(\cF_\cA)\right)\le nH_\cF- b.
$$
\end{lemma}

\begin{proof}
	Assume that $\mu(\cB_\infty) = 1$. Then for $\mu$-almost every $x\in\bM$ we have $\mu_x(\cB_\infty) = 1$. We fix such an $x$ and take a compact subset $K = K_x\subset \cB_\infty$ with $\mu_x(K) > 0.99$.
	
	Now, consider the finite measurable partition $\cC = \{K,K^c\}$. We have, for any $n\in\NN$,
	\begin{align*}
	&	H_{\mu_x}\left(\bigvee_{j=1}^{n}f^{-j}(\cF_\cA)\right)\\&\le H_{\mu_x}(\cC) + H_{\mu_x}\left(\bigvee_{j=1}^{n}f^{-j}(\cF_\cA)\bigg|\cC\right)\\
		&\le \log 2 + \mu_x(K) \cdot H_{\mu_x|_K}\left(\bigvee_{j=1}^{n}f^{-j}(\cF_\cA)\right) + \mu_x(K^c)\cdot H_{\mu_x|_{K^c}}\left(\bigvee_{j=1}^{n}f^{-j}(\cF_\cA)\right).
	\end{align*}	
	Similar to the proof of Lemma \ref{l.apbd}, we use the fact that $H_\nu(\cC)\le \log\#\cC_\nu$ for any probability measure $\nu$ and any partition $\cC$ with only finitely many $\nu$-positive measure set. 
	As before, we denote by 
	\begin{equation*}
		\begin{split}
			(\cF_\cA^n)_{\cF_\cA(x)} = \bigg\{A:\, &A \in \bigvee_{j=1}^{n}f^{-j}(\cF_\cA), A \cap \cF_\cA(x)\ne\emptyset\bigg\},
		\end{split}
	\end{equation*}
	which cuts $\cF_\cA(x)$ into finitely many disjoint intervals. We write
	$$
	(\cF_\cA^n)_{\cF_\cA(x)}^{K} = \left\{ A\in(\cF_\cA^n)_{\cF_\cA(x)}: A\cap K\ne\emptyset\right\}, 
	$$
	and similarly 
	$$
	(\cF_\cA^n)_{\cF_\cA(x)}^{K^c}  = \left\{ A\in (\cF_\cA^n)_{\cF_\cA(x)}: A\cap K^c\ne\emptyset\right\}.
	$$
	Then we must have 
	\begin{equation}\label{e.H1}
		\begin{split}
		&H_{\mu_x}\left(\bigvee_{j=1}^{n}f^{-j}(\cF_\cA)\right)\\
		&\le\log 2 + \mu_x(K) \cdot \log\#(\cF_\cA^n)_{\cF_\cA(x)}^K +  \mu_x(K^c)\cdot\log\#(\cF_\cA^n)_{\cF_\cA(x)}^{K^c}.		
		\end{split}
	\end{equation}	
	For $\#(\cF_\cA^n)_{\cF_\cA(x)}^{K^c}$ we use Lemma \ref{l.apbd}  (more precisely, Equation \eqref{e.a1}) to obtain 
	\begin{equation}\label{e.kc}
		\#(\cF_\cA^n)_{\cF_\cA(x)}^{K^c}\le \#(\cF_\cA^n)_{\cF_\cA(x)}\le  c_1e^{nH_\cF}, 
	\end{equation}
	where $c_1$ is some positive constant depending on $\cF$, $\cA$ and $f$ but not on $x$ or $n$.
	
	Below we control the cardinality of $(\cF_\cA^n)_{\cF_\cA(x)}^K$. Recall that $(f^{-j}(\cF_\cA))_{\cF_\cA(x)}$ are those elements of $f^{-j}(\cF_\cA)$ that intersects with $\cF_\cA(x)$.
	
	\medskip
	\noindent {\em Claim}. It holds that 
	$$
	\#(\cF_\cA^n)_{\cF_\cA(x)}^K\le 2 \sum_{j=1}^{n} \#\left\{A\in (f^{-j}(\cF_\cA))_{\cF_\cA(x)}: A\cap K\ne\emptyset \right\}.
	$$
	
	\begin{proof}[Proof of the claim]{
		The proof is similar to the proof of Lemma \ref{l.count} with only two differences. Firstly, if $A\in (f^{-k}(\cF_\cA))_{\cF_\cA(x)}$ does not intersect with $K$ then it will not create any element of $\#(\cF_\cA^k)_{\cF_\cA(x)}^{K}$. Secondly, if an endpoint $x$ meeting $A$ is associated with some element of $(f^{-j}(\cF_\cA))_{\cF_\cA(x)}$ (where $j<k$) that does not intersect $K$, then it will not create a new element of $(\cF_\cA^k)_{\cF_\cA(x)}^K$. This means that each $A\in (f^{-k}(\cF_\cA))_{\cF_\cA(x)}$ with $A\cap K\ne\emptyset$ creates $k'(A) + 1$ elements where $k'(A)$ is the total number of endpoints of those $B\in(f^{-j}(\cF_\cA))_{\cF_\cA(x)} $ with $B\cap K\ne\emptyset$. Then the claim follows by summing over $A$.
		}
%
%
%
		
	\end{proof}
	We continue our proof of Lemma \ref{l.w2}. Let $\nu\in\Gamma(\cF_\cA(x))$ be fixed. Due to the definition of $\cB_\ell$ and the fact that $K\subset \cB_\infty$, we must have $\nu(K)= 0$ as otherwise we get $\mu_x(K) > \ell\cdot \nu(K)$ for all $\ell \in \NN$, meaning $\mu_x(K) = \infty$ which is impossible. Let $\delta>0$ be arbitrary; since $K$ is compact, we can take $r =r(\delta)>0$ sufficiently small such that the set $U_\delta = B_{r(\delta)}(K)$ satisfies $\nu(U_\delta)<\delta.$ Since $\cF$ is an $f$-expanding foliation, there exists $N_1 = N_1(\delta)$ such that for any integer $j\ge  N_1(\delta)$, one has 
	$$
	A\in (f^{-j}(\cF_\cA))_{\cF_\cA(x)}^K\,\, \implies\,\, A\subset U_\delta.
	$$
	On the other hand, \eqref{e.exp} and \eqref{e.size} implies that  
	for every $A\in (f^{-j}(\cF_\cA))_{\cF_\cA(x)}$ (except for possibly two) it must holds that 
	\begin{equation}\label{e.n_0a}
		|n_0(A) - j|\le c_2
	\end{equation}
	for some constant $c_2>0$ only depending on $f$ and the partition $\cA$. 
	
	By Proposition \ref{p.ac4} and \eqref{e.n_0a}, for any $A\in (f^{-j}(\cF_\cA))_{\cF_\cA(x)}^K$,  it must hold that  
	$$
	\nu(A) \ge c_3 e^{-jH_\cF},
	$$
	where $c_3>0$ is a constant that does not depend on $\delta$ or $n$. 
	Together with $\nu(U_\delta)<\delta$ we obtain 
	\begin{align*}
	\#(f^{-j}(\cF_\cA))_{\cF_\cA(x)}^K&\le \frac{\delta}{c_3 e^{-jH_\cF} }\\
	& = \frac{\delta}{c_3} e^{jH_{\cF}}, \,\, \mbox{ whenever $j \ge N_1(\delta)$}.
	\end{align*}
	For the case $j< N_1(\delta)$ we have the coarse estimate similar to \eqref{e.a0} and \eqref{e.a1}: 
	$$
	\#(f^{-j}(\cF_\cA))_{\cF_\cA(x)}^K\le \frac{c_4}{a} e^{(j-n_0(\cF_\cA(x)))H_\cF},
	$$
	for all  $ j\in [n_0(\cF_\cA(x)),N_1(\delta))\cap\NN$. Here $c_4$ is a constant independent of $n$ and $\delta$.
	
	Combining the previous bounds, we obtain, for some $c_5 > 0$ independent of $n$ and $\delta$,
	\begin{align*}
		\#	(\cF_\cA^n)_{\cF_\cA(x)}^{K} &\le c_5 +  2\left(\sum_{j=n_0(\cF_\cA(x))}^{N_1(\delta)-1} + \sum_{j = N_1(\delta)}^{n}\right)\# (f^{-j}(\cF_\cA))_{\cF_\cA(x)}^K\\
		&\le c_5 + 2\sum_{j=n_0(\cF_\cA(x))}^{N_1(\delta)-1} \frac{c_4}{a} e^{(j-n_0(\cF_\cA(x)))H_\cF} + 2\sum_{j = N_1(\delta)}^{n} \frac{\delta}{c_3} e^{jH_\cF}\\
		&\le c_5 +c_6 e^{N_1(\delta)H_\cF} + \delta\cdot {c_7} e^{nH_\cF},
	\end{align*}
	where all constants $c_i$ are positive and do not depend on $n$ or $\delta.$ Note that for all $n\gg N_1(\delta)$  one has 
	$$
	e^{N_1(\delta)H_\cF} < \delta e^{nH_\cF}.
	$$
	This implies that for every $\delta>0$, there  exists $N_2(\delta)\gg N_1(\delta)$, such that for all $n>N_2(\delta)$ one has 
	\begin{equation}\label{e.a2}
		\#(\cF_\cA^n)_{\cF_\cA(x)}^K \le  {\delta}{c_8} e^{nH_\cF}
	\end{equation}
	where $c_8>0$ is a constant that does not depend on $\delta$ or $n$.
 
	Combining \eqref{e.H1}, \eqref{e.kc}, \eqref{e.a2} we obtain, for all $n>N_2(\delta),$
	\begin{align*}
		&H_{\mu_x}\left(\bigvee_{j=1}^{n}f^{-j}(\cF_\cA)\right)\\
		&\le \log 2 + \mu_x(K)\left(\log\delta +\log c_8 + nH_\cF \right) +\mu_x(K^c) \left(\log c_1 + nH_\cF\right) \\
		& = \Big(\log 2 +\mu_x(K)\cdot\log c_8 + \mu_x(K^c)\cdot  \log c_1\Big) + \big(\mu_x(K) + \mu_x(K^c)\big) nH_\cF\\
		&\hspace{1cm} + \mu_x(K)\log\delta\\
		& =  c_9 + nH_\cF +\mu_x(K)\log\delta,
	\end{align*}
	where $c_9$ is a constant that does not depend on $n$ or $\delta$ but possibly depend on $x$ (since $K$ depends on $x$).
	
	Recall that $\mu_x(K) > 0.99$ and the choice of $\delta$ is arbitrary. In particular, given any $b>0$ we may choose $\delta = \delta(x,b)$ close to zero such that 
	$$
	\log c_9 + 0.99\log\delta<-b. 
	$$
	It follows that for such $\delta$ and every $n > N_2(\delta) = N_2(x,b)$ one has 
	$$
	H_{\mu_x}\left(\bigvee_{j=1}^{n}f^{-j}(\cF_\cA)\right)<nH_\cF-b,
	$$
	finishing the proof of Lemma \ref{l.B}.
\end{proof}

Now we are ready to show that $\mu_x\prec \nu$ for $\mu$-almost every $x$ and any $\nu\in\Gamma(\cF_\cA(x))$. Note that 
$$
H_\mu\left(\bigvee_{j=1}^n f^{-j}(\cF_\cA)\Big|\cF_\cA\right) =\int H_{\mu_x}\left(\bigvee_{j=1}^n f^{-j}(\cF_\cA)\right)\,d\mu(x).
$$
Similar to Lemma \ref{l.apbd}, here the conditioning is dropped for the integrand because each $\mu_x$ is supported on the partition element $\cF_\cA(x)$. 

Let $b=100$. We choose $\delta>0$ sufficiently small such that
\begin{equation}\label{e.delta1}
	-100(1-\delta) + \delta L_2<0
\end{equation}
where $L_2>0$ is given by Lemma \ref{l.apbd} independent of $x$ and $n$. Then, we select a subset $K_\delta$ consisting of $\mu$-typical points, such that (1) $\mu(K_\delta)>1-\delta$, and (2) for every $x\in K_\delta$, the constant $N= N(x,100)$ given by Lemma \ref{l.B} satisfies $N\le  N_0$ for some constant $N_0$ sufficiently large. Here we increase $N_0$ to make it measurable w.r.t.\ $x$. For any $n>N_0$ we write, 
using Proposition \ref{p.ac4}, Lemma \ref{l.apbd} and Lemma \ref{l.B}, 
\begin{align*}
H_\mu\left(\bigvee_{j=1}^n f^{-j}(\cF_\cA)\Big|\cF_\cA\right)=	&\int H_{\mu_x}\left(\bigvee_{j=1}^n f^{-j}(\cF_\cA)\right)\,d\mu(x)\\
	&= \left(\int_K + \int_{K^c}\right) H_{\mu_x}\left(\bigvee_{j=1}^n f^{-j}(\cF_\cA)\right)\,d\mu(x)\\
	&\le \int_K (nH_\cF -100)d\mu + \int_{K^c} (nH_\cF+L_2)\ d\mu\\
	&= nH_\cF -100\mu(K) + L_2\mu(K^c)\\
	&\le nH_\cF - 100(1-\delta) + \delta L_2\\
	&< nH_\cF,
\end{align*}
where the last inequality is due to \eqref{e.delta1}.  However, Lemma \ref{l.mono} shows that 
$$
H_\mu\left(\bigvee_{j=1}^n f^{-j}(\cF_\cA)\Big|\cF_\cA\right)\ge nh_\mu^\cF(f) = nH_\cF
$$
for any $n\in\NN$, which is a contradiction. We conclude that $\mu(\cB_\infty) = 0$. This implies that $\mu_x\prec \nu$ for $\mu$-almost every $x$ and any $\nu\in\Gamma(\cF_\cA(x))$.  For otherwise, one can obtain a positive $\mu$-measure set $B$ such that for any $x\in B$ one must have $\mu_x(B)>0$ and $\nu(B) = 0$ for some $\nu\in\Gamma(\cF_\cA(x))$; then $\mu_x$-almost every point in $B$ must belong to $\cB_\infty$ due to the Lebesgue-Besicovitch Density Theorem (see, for instance, \cite[Theorem 1.32]{EG}) and Proposition \ref{p.ac4}. This shows that $\mu(\cB_\infty)>0$, a contradiction.

\subsection{Finishing the proof of Theorem \ref{t.1}}
To simplify notation, we will continue writing $\mu_x = \mu_x^{\cF,\cA}$. 
The goal of this section is to prove the following:
\begin{proposition}\label{p.2}
	Under the assumptions of Theorem \ref{t.1}, for $\mu$-almost every $x\in\bM$, the conditional measure $\mu_x$ satisfies the Gibbs property: there exists $L_0>1$ such that for any $\cF$-disk $I\subset \cF_\cA(x)$, one has
	$$
	\mu_x(I)\asymp_{L_0} e^{-n_0(I)H_\cF}.
	$$
\end{proposition}
Indeed we will see that $L_0 = L^4\cdot\widetilde L\cdot e^{H_\cF}$  where $\widetilde L$ is given by Proposition \ref{p.ac4}.  It is easy to see that Proposition \ref{p.2}, together with Theorem \ref{t.properties} (1), imply Theorem \ref{t.1} with $L_1 = L\cdot L_0$.

The rest of this section is dedicated to the proof of Proposition \ref{p.2}.  We adapt a classical argument due to Sinai \cite{Sinai} for SRB measures of hyperbolic diffeomorphisms.

We take a $\mu$-typical point $x\in\bM$ and recall that $\cA(x)$ is contained in a foliation box $\cB_k$ (see the discussion at the beginning of Section \ref{ss.partition}). To simplify notation we may assume that $\cA(x)\subset \cB_1$. We also let the $\cF$-disk $I\subset \cF_\cA(x)$ be fixed.

We also fix any point $y\in \bM$ for which $\mu_y$ is well-defined and $\mu_y\prec \nu_y$ for any $\nu_y\in\Gamma(\cF_\cA(y))$, thanks for the result of Section \ref{ss.3.2}. On $\cF_\cA(y)$ we fix a compact set $K = K(y)$ with positive $\mu_y$ measure such that for every $z\in K:$
\begin{itemize}
	\item $z$ is a typical point of $\mu$ in the sense of Birkhoff: one has 
	$$
	\frac{1}{n}\sum_{k=0}^{n-1} \delta_{f^k(z)} \to  \mu.
	$$ 
	\item $\mu_z$ is well-defined and coincides with $\mu_y.$
\end{itemize}
Such a $K$ exists because points on $\cF_\cA(y)$ satisfying both requirements have full $\mu_y$-measure. Also note that $\nu_y(K)>0$ for any $\nu_y\in\Gamma(\cF_\cA(y))$ since $\mu_y\prec\nu_y$ by the previous subsection. We fix such a $y.$

Next, we consider the following probability measures:
$$
w_n:= \frac{1}{n}\sum_{k=0}^{n-1} (f^k)_*(\nu_y|_{K}),
$$
where $\nu_y|_K = \frac{\nu_y(\cdot\,\cap K)}{\nu_y(K)}$ is the conditional measure of $\nu_y$ on $K$.

\begin{lemma}\label{l.w1}
	$w_n$ converges in weak-* topology to $\mu$.
\end{lemma}

\begin{proof}
	Let $\phi:\bM\to\RR $ be any continuous function. We write
	\begin{align*}
		w_n(\phi) &= \int\phi\ d\left(\frac{1}{n}\sum_{k=0}^{n-1} (f^k)_*(\nu_y|_{K})\right)\\
		& = \frac1n\sum_{k=0}^{n-1} \int \phi\circ f^{k} \ d\nu_y|_{K}\\
		& = \int \left( \frac1n\sum_{k=0}^{n-1} \phi\circ f^k(x) \right) \ d\nu_y|_K(x).
	\end{align*}
Note that $\frac1n\sum_{k=0}^{n-1} \phi\circ f^k(y) $ converges to $\mu(\phi)$ pointwise  on $K$. By the dominated convergence theorem we see that $w_n(\phi)$ must converge to $\nu_y|_K(\mu(\phi))= \mu(\phi).$ Since $\phi$ is arbitrary, we conclude that $w_n$ converges to $\mu$ in weak-* as required.
\end{proof}

Next we approximate $K$ by a sequence of open (for the relative topology on $\cF_\cA(y)$) neighborhoods $U_n$ in the following way: 
\begin{equation}\label{e.U}
U_n := \bigcup_{z\in f^{n}(\cF_\cA(y)): \cF_\cA(z)\cap f^n(K)\ne\emptyset} f^{-n}(\cF_\cA(z)).
\end{equation}
The construction implies $K\subset U_n$. Since the diameter of  sets of the form $f^{-n}(\cF_\cA(z))$ are exponentially small in $n$ (uniformly in $z$) and $K$ is compact, we see that $\bigcap_{n\ge 0} U_n = K,$ and hence 
$$
\nu_y(U_n)\to \nu_y(K).
$$ 

Next, we define a new sequence of probability measures:
\begin{equation}\label{e.tw}
	\tilde w_n:= \frac{1}{n}\sum_{k=0}^{n-1} (f^k)_*(\nu_y|_{U_k})
\end{equation}
where, as before, $\nu_y|_{U_k}$ is the conditional measure of $\nu_y$ on $U_k$.

\begin{lemma}\label{l.w2}
	$\tilde w_n$ converges in weak-* topology to $\mu$. 
\end{lemma}
\begin{proof}
	Let $\phi:\bM\to\RR$ be any continuous function. We consider 
	\begin{align*}
		&|(f^k)_*(\nu_y|_K)(\phi) - (f^k)_*(\nu_y|_{U_k})(\phi)|\\ &= \left|\frac{\nu_y(\phi\circ f^{k}|_ K)}{\nu_y(K)} - \frac{\nu(\phi\circ f^{k}|_{U_k})}{\nu_y(U_k)}\right|\\
		&= \left|\frac{\nu_y(\phi\circ f^{k}|_ K)}{\nu_y(K)} - \left(\frac{\nu_y(\phi\circ f^{k}|_{K})}{\nu_y(K)} + \frac{\nu_y(\phi\circ f^{k}|_{U_k\setminus K})}{\nu_y(K)}\right)\frac{\nu_y(K)}{\nu_y(U_k)}\right|\\
		& \le \frac{\nu_y(|\phi|\circ f^{k}|_ K)}{\nu_y(K)}\left(1 - \frac{\nu_y(K)}{\nu_y(U_k)}\right) + \frac{\nu_y(|\phi|\circ f^{k}|_{U_k\setminus K})}{\nu_y(K)}\frac{\nu_y(K)}{\nu_y(U_k)}\\
		&\le {|\phi|_{C^0}}\left(1 - \frac{\nu_y(K)}{\nu_y(U_k)} +  \frac{\nu_y(U_k\setminus K)}{\nu_y(U_k)}\right) ,
	\end{align*}
	which converges to $0$ as $k\to\infty$. This shows that $\tilde w_n$ and $w_n$ must have the same limit, which has to be $\mu$  due to Lemma \ref{l.w1}. 
\end{proof}

We recall the construction of the conditional measures $\mu_x$ by Rokhlin.
Let $\Phi:\cB_1\to \RR^{\dim\bM} = \RR\times \RR^{\dim\bM-1}$ be a continuous map that sends each local leaf of $\cF$ to a line parallel to the first coordinate axis $\RR$.
By our assumptions on $\cF$, $\Phi$ restricted to each local leaf of $\cF$ is a $C^{1+\alpha}$ diffeomorphism. Without loss of generality, we assume that $\cF_{loc}(x)$, the local leaf in $\cB_1$ containing $x$, is sent to the first coordinate axis $\RR$ and $\Phi(x)$ is the origin. For each $r>0$, we enlarge the first coordinate axis $\RR$ by $r$ to obtain a cylinder 
$$
B_r(\RR) := \RR\times B_r^{\RR^{\dim\bM-1}}(0),
$$ 
where $B_r^{\RR^{\dim\bM-1}}(0)$ is the ball in $\RR^{\dim\bM-1}$ centered at the origin with radius $r$. 
We then denote by $\mathbf B_r(\cF_\cA(x))$ the set $\Phi^{-1} \left(B_r(\RR)\right)\cap \cA(x).$ $\mathbf B_r(\cF_\cA(x))$ can be seen as an open (for the relative topology on $\cA(x)$) tubular neighborhood of $\cF_\cA(x)$ with size $r$, consists entirely of local leafs $\cF_\cA(z)$, and shrinks to $\cF_\cA(x)$ as $r\to 0$. Since $x$ is assumed to be a typical point of $\mu$, one must have $\mu(\mathbf B_r(\cF_\cA(x)))>0$ for any $r>0.$

Recall that $I\subset \cF_\cA(x)$ is an $\cF$-disk that has been fixed. We define the cylinder around $I$ of size $r$ by 
$$
I_r: = \Phi^{-1} \left(\Phi(J)\times  B_r^{\RR^{\dim\bM-1}}(0) \right)\subset \mathbf B_r(\cF\cA(x)).
$$
By Rokhlin's construction of conditional measures (or by using the Martingale Convergence Theorem), we see that the conditional measure $\mu_x$ is the weak-* limit of $\mu$ conditioned on the cylinder $I_r.$ More precisely, 
\begin{equation}\label{e.conditional}
\mu_x(I) = \lim_{r\to0}\mu|_{\mathbf B_r(\cF_\cA(x))}(I_r). 	
\end{equation}

{
}

The next lemma is the key step in proving Proposition \ref{p.2}.

\begin{lemma}\label{l.w3}
	Let $L_0 = L^4 \widetilde L e^{H_\cF}$. Then, for every $r>0$ sufficiently small (depending on $I$) and $n\in\NN$ sufficiently large, it holds that
	$$
	\tilde w_n|_{\mathbf B_r(\cF_\cA(x))}(I_r) \asymp_{L_0} e^{-n_0(I)H_\cF}.
	$$ 
\end{lemma}
\begin{proof}
	Recall from \eqref{e.tw} that $\tilde w_n= \frac{1}{n}\sum_{k=0}^{n-1} (f^k)_*(\nu_y|_{U_k})$. For each $k$, ${U_k}$ is a finite disjoint union of $U_k^j = f^{-k}(\cF_\cA(z_j))$ (recall \eqref{e.U}), each of which is an $\cF$-disk. This means that $\tilde w_n|_{\mathbf B_r(\cF_\cA(x))}$ is a finite convex combination of measures of the form $(f^k)_*(\nu_y|_{U_k^i})$.
 	 By Theorem \ref{t.properties} (3), we have 
	\begin{equation}\label{e.w3a}
		\frac{d\nu_y|_{U_k^j} }{d\nu_k^j}\asymp_{L^4}1
	\end{equation}
	for every $j$ and any $\nu_k^j\in\Gamma(U_k^j)$. Also note that  $$
	(f^k)_*(\nu_k^j) \in (f^k)_*(\Gamma(U_k^j)) =  (f^k)_*(\Gamma(f^{-k}(\cF_\cA(z_j))))  = \Gamma(\cF_\cA(z_j))
	$$ 
	by Theorem \ref{t.properties} (1).
	
	On the other hand, each $\cF_\cA(z_j)$ (except for possibly two associated to the endpoints of $f^k(\cF_\cA(y))$) either lies entirely in $\mathbf B_r(\cF_\cA(x))$, or is disjoint with it; only the former counts towards $\tilde w_n|_{\mathbf B_r(\cF_\cA(x))}$. For those lying in $\mathbf B_r(\cF_\cA(x))$, $I_r\cap \cF_\cA(z_j)$ is an $\cF$-disk of length approximately the same as $I$ (as long as $r$ is sufficiently small). With $I$ fixed, we may shrink $r$ if necessary to obtain that
	$$
	|n_0(I_r\cap \cF_\cA(z_j)) - n_0(I_r)|\le 1.
	$$
	Then, we use Proposition \ref{p.ac4} to obtain 
	\begin{equation}\label{e.w3b}
		\nu(I_r)\asymp_{\widetilde L} e^{-n_0(I_r\cap\cF_\cA(z_j))H_\cF} \asymp_{\widetilde L\cdot e^{H_\cF}}   e^{-n_0(I)H_\cF}
	\end{equation}
	for every $\nu\in\Gamma(\cF_\cA(z_j))$. In particular, \eqref{e.w3b} holds for $	(f^k)_*(\nu_k^j)$.
	Combining \eqref{e.w3a} and \eqref{e.w3b}, we have
	\begin{equation}\label{e.w3c}
		(f^k)_*(\nu_y|_{U_k^j})(I_r)\asymp_{L^4} (f^k)_*(\nu_k^j)(I_r)\asymp_{L_0}  e^{-n_0(I)H_\cF},
	\end{equation}
	where $L_0 = L^4\widetilde Le^{H_\cF}.$
	
	Next, we note that $\tilde w_n|_{\mathbf B_r(\cF_\cA(x))}$ is a finite convex combination of $(f^k)_*(\nu_y|_{U^j_k})$. Therefore, \eqref{e.w3c} leads to 
	$$
	\tilde w_n|_{\mathbf B_r(\cF_\cA(x))}(I_r) \asymp_{L_0}  e^{-n_0(I)H_\cF}.
	$$
	This finishes the proof of Lemma \ref{l.w3}.
\end{proof}

\begin{proof}[Proof of Proposition \ref{p.2}]
	Combining Lemma \ref{l.w2} and Lemma \ref{l.w3} we have that 
	$$
	\mu|_{\mathbf B_r(\cF_\cA(x))}(I_r) \asymp_{L_0}  e^{-n_0(I)H_\cF}.
	$$
	Then, Equation \ref{e.conditional} shows that 
	$$
	\mu_x(I)\asymp_{L_0}  e^{-n_0(I)H_\cF}.
	$$
	This finishes the proof of Proposition \ref{p.2}.
\end{proof}

Now Theorem \ref{t.1} is a direct consequence of Proposition \ref{p.2} and the Gibbs property of $\nu$ (Theorem \ref{t.properties}  (2)). It also follows from the Gibbs property that the support of any ergodic measure of maximal $\cF$-entropy must consist of entire $\cF$-leaves.

\section{Rigidity}
In this section we consider the case when the weak Margulis measures in $\Gamma(D)$ are equivalent to the leaf volume. Under this assumption we obtain rigidity: the log Jacobian of $f$ along $\cF$ must be cohomologous to the constant $H_\cF$, the topological entropy along $\cF$.  
\begin{theorem}\label{t.2}
	Let $f$ be a $C^{1+\alpha}$ diffeomorphism preserving an expanding foliation $\cF$ with homogeneous growth. Assume that $\mu$, an ergodic measure of maximal $\cF$-entropy, is also a Gibbs $\cF$-state of $f$. Then there exists a $\mu$ measurable function $\phi$ satisfying the cohomological equation
	\begin{equation}\label{e.ch}
		\log |D_xf|_{T_x\cF}|  = \phi\circ f(x) - \phi(x) + H_\cF 
	\end{equation} 
	at $\mu$ almost every point. 
\end{theorem}

\subsection{Constructing a measurable function $\phi$}

For definiteness of conditional measures of $\mu$, we will keep using the measurable partition $\cF_\cA$ constructed in Section \ref{s.3}.
Recall that $\cB_{k}, k=1,\cdots, N_0$ is a collection of foliation boxes covering $\bM$, and each local leaf $\cF_{loc}(x)$, the connected component of $B_i\cap \cF(x)$ containing $x$, is a one-dimensional segment and therefore orientable (although $f$ may not preserve this orientation). 


Now let $\mu$ be an ergodic measure of maximal $\cF$-entropy which is also a Gibbs $\cF$-state. This means that the conditional measures of $\mu$ on $\cF$-disks are absolutely continuous with respect to the leafwise Lebesgue measure $\Leb_{\cF_\cA(x)}$, and  the density $\frac{d\mu_x}{d\Leb_{\cF_\cA(x)}}$ is positive for $\mu$-almost every $x$. By Theorem \ref{t.1} we have that measures in $\Gamma(\cF_\cA(x))$, at typical points of $\mu$, are also absolutely continuous with respect to $\Leb_{\cF_\cA(x)}$.
This invites us to define the following function:
	$$
	\phi(x): = \lim_{n\to 0}\sup_{x\in B, |B|_\cF<1/n}  \log\left(\frac{|B|_\cF}{e^{-n_0(B)H_\cF}}\right),
	$$
where the supremum is taken over all open $\cF$-disks $B$, containing $x$, with diameter less than $1/n$. The limit exists because of monotonicity. 	
 We claim that $\phi$ is $\mu$-measurable. The proof is similar to the proof of the universal measurability of $\mathcal B_\ell$ in Section \ref{s.3}. We set 
	$$
	F_n(x) = \sup_{x\in B, |B|_\cF<1/n}  \log\left(\frac{|B|_\cF}{e^{-n_0(B)H_\cF}}\right),
	$$
and note that each $\cF$-disk $B$ containing a point $y\in\bM$ can be uniquely identified as a triple $B_{y,r_1,r_2} := (y, r_1,r_2)$ with $r_1,r_2\in \RR^+$ being the distance between $y$ and the left and right endpoint of $B$. This invites us to define
$$
h(y,r_1,r_2) =  \log\left(\frac{|B_{y,r_1,r_2}|_\cF}{e^{-n_0(B_{y,r_1,r_2})H_\cF}}\right) = \log\left(\frac{r_1+r_2}{e^{-n_0(B_{y,r_1,r_2})H_\cF}}\right).
$$
It is easy to see that $h:\bM\times\RR\times\RR \to\RR$ is Borel measurable. Let $\pi_1:\bM\times\RR\times\RR\to \RR$ be the projection to the first coordinate. Then, 
$$
\{F_n(x)>a\} = \pi_1\left\{(y,r_1,r_2): r_1 + r_2 <1/n, h(y,r_1,r_2) >a \right\}
$$
is an analytic set and therefore $\mu$-measurable. It follows that the function $F_n(x)$ is $\mu$-measurable for every $n$, and the same holds for $\phi.$


Next we show that $\phi$ is $\mu$-almost everywhere finite. To this end we take any $\delta>0$ and choose a compact set $K_\delta$ such that $\mu(K_\delta)> 1-\delta$, and a constant $C_\delta>0$ such that. 
\begin{equation}\label{e.bdd1}
\frac{d\mu_x}{d\Leb_{\cF_\cA(x)}}\ge C_\delta>0, \,\, \forall x\in K_\delta.
\end{equation}
 Recall from Theorem \ref{t.1} that at $\mu$ almost every point $x$, the conditional measure $\mu_x : = \mu_x^{\cF,\cA}$ is equivalent to any reference measure $\nu\in\Gamma(\cF_\cA(x))$ with Radon-Nikodym satisfying $\frac{d\mu_x}{d\nu}\asymp_{L_1} 1$ where $L_1>1$ is a constant. Combining this with \eqref{e.bdd1}, we see that $\frac{d\Leb_{\cF_\cA(x)}}{d\nu}\le L_1 C_\delta^{-1}$ for $x\in K$. The Lebesgue-Besicovitch Density Theorem yields 
 	$$
 	\lim_{|I|_\cF\to 0}\,\frac{|I|_\cF}{\nu(I)} \le L_1 C_\delta^{-1},
 	$$ 	
 where $I$ is any $\cF$-disk containing $x$.	
 It then follows from the Gibbs property of $\nu$ (Proposition \ref{p.ac4}) that for every $\cF$-disk $B\ni x$ sufficiently small, the function
$$
\psi(x) = \frac{|B|_\cF}{e^{-n_0(B)H_\cF}}
$$
must be bounded by the constant $2C_\delta^{-1} \cdot L_1\cdot \tilde L$ where $\tilde L> 1$ is the constant given by Proposition \ref{p.ac4}. This shows that $\phi(x)$ is bounded on $K_\delta$ for every $\delta>0$; therefore $\phi$ is  finite almost everywhere.

\subsection{$\phi$ solves the cohomological equation}
Below we show that $\phi$ is a solution to the cohomological equation \eqref{e.ch}. 
To this end, for $\mu$-almost every $x\in \bM$ we write 
$$
\lambda_x : =| D_xf|_{T_x\cF}|.
$$
We fix $x$ and take a sequence of $\cF$-disks $B_n\ni x$ with $|B_n|_\cF <1/n$ and 
$$
\log\left(\frac{|B_n|_\cF}{e^{-n_0(B_n)H_\cF}}\right) > F_n(x) - 1/n.
$$
This means that  
\begin{equation}\label{e.rn}
\phi(x) = \lim_{n\to \infty} \log\left(\frac{|B_{n}|_\cF}{e^{-n_0(B_n)H_\cF}}\right).
\end{equation}

At $y = f(x)$ we 
define
\begin{equation*}\label{e.tp}
\widetilde \phi(y) = \lim_{n\to \infty} \log\left(\frac{|f(B_{n})|_\cF}{e^{- n_0(f(B_n))H_\cF}}\right),
\end{equation*}
where the sequence of disks $(B_n)$ are chosen as $x = f^{-1}(y)$ as before.

\begin{lemma}\label{l.r1}
	The limit defining $\widetilde\phi(y)$ exists and coincides with $\phi(x) + \log\lambda_x - H_\cF$. In particular, $\widetilde\phi$ is $\mu$-measurable and independent of the choice of the sequence $(B_n)$.
\end{lemma}
\begin{proof}
	Note that $n_0(f(B_n))= n_0(B_n) -1$. We write 
	\begin{align*}
		\frac{|f(B_n)|_\cF}{e^{- n_0(f(B_n))H_\cF}} &=  \frac{|f(B_n)|_{\cF}}{|B_n|_\cF} \cdot  \frac{|B_n|_\cF}{e^{- n_0(f(B_n))H_\cF}}\\
		&= \frac{|f(B_n)|_{\cF}}{|B_n|_\cF} \cdot  \frac{|B_n|_\cF}{e^{ -n_0(B_n) H_\cF}}\cdot e^{-H_\cF}.
	\end{align*}
	The first term converges to the unstable Jacobian of $f$ at $x$, which is $\lambda_x$; the second term converges to $e^{\phi(x)}$ due to our choice of $B_n$ (see \eqref{e.rn}). The lemma follows. 
\end{proof}

\begin{lemma}\label{l.r2}
	At $\mu$-almost every $x$ we have 
	$$
	\widetilde\phi(f(x)) = \phi(f(x)). 
	$$
\end{lemma}
\begin{proof}
	First we will show that for any sequence of $\cF$-disks $\widetilde B_n\ni f(x)$ with $|\widetilde B_n|_\cF<1/n$ such that the limit 
	\begin{equation}\label{e.adef}
	A = A(\{\widetilde B_n\}):=\lim_{n\to\infty} \frac{|\widetilde B_{n}|_\cF}{e^{-n_0(\widetilde B_n)H_\cF}}	
	\end{equation}
	exists, one must have 
	\begin{equation*}\label{e.A}
		A\le \widetilde\phi(f(x)). 
	\end{equation*}
	The proof is divided into the following two claims.
		
	\medskip 
	\noindent {\em Claim 1.}
	The limit 
	\begin{equation}\label{e.ta}
		\widetilde A = \lim_{n\to \infty} \log\left(\frac{|f^{-1}(\widetilde B_n)|_\cF}{e^{-n_0(f^{-1}(\widetilde B_n))H_\cF}}\right).
	\end{equation}
	exists and satisfies $\widetilde A = A - \log\lambda_x + H_\cF.$
	
	\begin{proof}[Proof of Claim 1]
		The proof is essentially the same as the proof of Lemma \ref{l.r1} but in the reserved direction. We write
		\begin{align*}
			\frac{|f^{-1}(\widetilde B_n)|_\cF}{e^{-n_0(f^{-1}(B_{\tau_n,*}^\cF(f(x))))H_\cF}} &=  \frac{|f^{-1}(\widetilde B_n)|_{\cF}}{|\widetilde B_n|_\cF} \cdot  \frac{|\widetilde B_n|_\cF}{e^{-n_0(f^{-1}(\widetilde B_n))H_\cF}}\\
			&= \frac{|f^{-1}(\widetilde B_n)|_{\cF}}{|\widetilde B_n|_\cF}  \cdot  \frac{|\widetilde B_n|_\cF}{e^{ -n_0(\widetilde B_n) H_\cF}}\cdot e^{H_\cF};
		\end{align*}
	here to obtain the second equality, we use the fact that $n_0(f^{-1}(\widetilde B_n)) = n_0(\widetilde B_n) + 1.$
	
	Note that as $|B_n|_\cF\to 0,$ the first term converges to $|D_{f(x)}f^{-1}|_{T_{f(x)}\cF}|$ which is $1/\lambda_x$. The second term converges to $e^A$ by \eqref{e.adef}. This shows that $e^{\widetilde A} = e^{A+H_\cF}/\lambda_x$, and the claim follows by taking logarithm. 
	
	\end{proof}

	\medskip 
	\noindent {\em Claim 2.} For any sequence $\{\widetilde B_m\}$ containing $f(x)$ with $|\widetilde B_m|_\cF <1/m$ such that the limit $\widetilde A$ defined by \eqref{e.ta} exists, we have $\widetilde A\le \phi(x).$
	
	\begin{proof}[Proof of Claim 2.]
		We observe that $f^{-1}(\widetilde B_n)$ is an $\cF$-disk containing $x$. Since $f$ is uniformly expanding on each leaf of $\cF$, we have $|f^{-1}(\widetilde B_n)|_\cF < |\widetilde B_n|_\cF < 1/n$. By the definition of $\phi(x)$, we have  
		$$
		\widetilde A = \lim_{n\to\infty} \log\left(\frac{|f^{-1}(\widetilde B_n)|_\cF}{e^{-n_0(f^{-1}(\widetilde B_n)) H_\cF}}\right) \le \lim_{n\to\infty } F_n(x) = \phi(x),
		$$
		as claimed. 
	\end{proof}
	
	Combining Claim 1, Claim 2 and Lemma \ref{l.r1}, we obtain
	$$
	A-\log\lambda_x + H_\cF = \widetilde A \le \phi(x) = \widetilde\phi(f(x)) -\log\lambda_x +H_\cF,
	$$
	that is,  $A\le \widetilde \phi(f(x))$. Since the sequence $\{\widetilde B_n\}$ in the definition of $A$ were taken arbitrarily, we conclude that 
	\begin{equation}\label{e.ge}
		\phi (f(x))\le \widetilde \phi (f(x)).
	\end{equation}
	
	It remains to derive the reversed inequality. The proof is similar to the proof of Claim 2. Note that for each $\cF$-disk $B\ni x$ with $|B|_\cF < 1/n,$ $f(B)$ is an $\cF$-disk containing $f(x)$ with $|f(B)|_\cF< 2\lambda_x/n$ (as long as $n$ is sufficiently large). This shows that 
	\begin{equation}\label{e.le}
		\tilde\phi(f(x)) =\lim_{n\to\infty} \log \frac{|f(B_n)|_\cF}{e^{n_0(f(B_n))H_\cF}}\le \lim_{n\to\infty}F_{\lceil 2\lambda_x n\rceil}(x) =  \phi(f(x)).		
	\end{equation}
	Here $\lceil a\rceil$ denotes the smallest integer greater than $a.$ 
	Combining \eqref{e.ge} and \eqref{e.le} yields $\tilde\phi(f(x)) = \phi(f(x))$. The proof of Lemma \ref{l.r2} is now complete.
\end{proof}

Combining Lemma \ref{l.r1} and \ref{l.r2}, we obtain
$$
\phi(f(x)) = \phi(x)+\log\lambda_x - H_\cF,
$$
that is, $\phi$ is a $\mu$-measurable solution to the cohomological equation \eqref{e.ch}. This concludes the proof of Theorem \ref{t.2}.

\section{Transversal holonomy}\label{s.holonomy}
In this section we consider the absolute continuity of holonomy maps induced by a transversal foliation $\widetilde\cF$. For the sake of clarity, we will write $\cF_r(x)$ instead of $B_r^\cF(x)$ in previous sections for the $\cF$-disk of radius $r$ centered at $x$; $\widetilde\cF_r(x)$ is defined similarly. 

\begin{definition}\label{d.TF}
	Given an expanding foliation $\cF$ with one-dimensional leaves, we say that a foliation $\widetilde \cF$ is a transversal foliation, if:
	\begin{enumerate}
		\item (Regularity) $\widetilde\cF$ is a continuous foliation with $C^1$ leaves. 
		\item (Invariance) $f$ maps $\widetilde\cF$-leaves to $\widetilde\cF$-leaves.
		\item (Transversality) $\widetilde \cF$ is transversal to $\cF$ at every point.
	\end{enumerate}
\end{definition}
It follows that $\widetilde\cF$-leaves are $(\dim\bM-1)$-dimensional. Furthermore, $\widetilde\cF$ and $\cF$ form a {\em (topological) local product structure}: there exists $\delta>0$ such that for every $x\in \bM$, every $y\in \cF_\delta(x)$ and $z\in \widetilde\cF_\delta(x)$ satisfies that $\widetilde\cF_{4\delta} (y)$ and $\cF_{4\delta}(z)$ intersection transversely at a unique point, which we will denote by $[y,z].$ Furthermore, this map is surjective in the sense that every point in a small neighborhood of $x$ can be written uniquely as $[y,z]$ for some $y\in \cF_\delta(x)$ and $z\in \widetilde\cF_\delta(x)$.

Next we consider local holonomy maps induced by a transversal foliation. Let $\delta>0$ be as above. Given $x\in\bM$ and $y,z\in \widetilde\cF_\delta(x)$, we define the holonomy $H^{loc}_{y,z}:\cF_\delta(y)\to \cF_\delta(z)$ as the map that sends each point $p\in \cF_\delta(y)$ to the unique transversal intersection between $\widetilde \cF_{4\delta}(p)$ and $\cF_{4\delta}(z)$. The goal of this section is to show that under proper conditions, $(H^{loc}_{y,z})_*$ is absolutely continuous with respect to measures in $\Gamma (\cF_{\delta}(y))$ and $\Gamma( \cF_{\delta}(z))$, with uniformly controlled Jacobian.

\begin{definition}\label{d.NE}
	 We say that $f$ is not expanding along a transversal foliation $\widetilde \cF$, if there exists two positive constants $\delta_1 < R$ such that for every $x\in \bM$ and $n\in \NN$, it holds that 
	 $$
	 f^n(\widetilde\cF_{\delta_1}(x))\subset \widetilde\cF_{R}(f^n(x)).
	 $$ 
\end{definition}
Note that $R$ is possibly very large (even when compared to the diameter of the manifold). Being not expanding is weaker than Lyapunov stability (in which case $R$ must be arbitrary), and is satisfied by many systems. One of them is the perturbation of time-one maps of geodesic flows, as we will see in Section \ref{s.geodesic}.

\begin{theorem}\label{t.holonomy}
	Let $\cF$ be an expanding foliation for a $C^{1+\alpha}$ diffeomorphism $f$ with homogeneous growth, and $\widetilde\cF$ is a transversal foliation along which $f$ is not expanding. Then there exists $L_H>1$ such that for all $x\in\bM,$ $y,z\in \widetilde\cF_\delta(x)$ and every $\nu_y\in \Gamma(\cF_\delta(y))$, it holds that $(H^{loc}_{y,z})_*\nu_y$ is absolutely continuous with respect to any $\nu_z\in \Gamma(\cF_\delta(z))$, with 
	$$
	\frac{d(H^{loc}_{y,z})_*\nu_y}{d\nu_z}\asymp_{L_H} 1.
	$$ 
\end{theorem}

The proof of this theorem occupies the rest of this section. We shall assume from now on that $\delta < \delta_1$, since a local product structure at scale $\delta$ implies the same structure at smaller scales. 

\begin{proof}[Proof of Theorem \ref{t.holonomy}]
	First we recall Remark \ref{r.vep}: instead of requiring $|D|_\cF\ge 1$ when defining homogeneous growth and more importantly $n_0$, we can ask for $|D|_\cF\ge \vep$ for any fixed $\vep>0$. Here we will do so for a fixed $\vep = 2\delta$ where $\delta$ is the scale of the local product structure. 
	
	Let $x,y,z$ be as in the theorem, and note that without loss of generality we may take $z=x$ since $H^{loc}_{y,z} = H^{loc}_{x,z}\circ H^{loc}_{y,x} $. To simplify notation we write $D_w = \cF_\delta(w)$ for any $w\in\widetilde \cF_\delta(x)$. Fix any small $\cF$-disk $I_y\subset \cF_\delta(y)$ and denote by 
	$I_w = H_{y,w}(I_y)\subset \cF_\delta(w)$ for any $w\in \widetilde \cF_\delta(x)$. By the Gibbs property (Theorem \ref{t.properties} (2)), any $\nu_y\in \Gamma(\cF_\delta(y))$ and $\nu_x\in \Gamma(\cF_\delta(x))$ satisfy 
	$$
	\nu_y(I_y)\asymp_L\frac{e^{-n_0(I_y)H_\cF}}{e^{-n_0(D_y)H_\cF}},\,\, \nu_x(I_x)\asymp_L\frac{e^{-n_0(I_x)H_\cF}}{e^{-n_0(D_x)H_\cF}}
	$$
	Note that $|D_y|_\cF = |D_x|_\cF = 2\delta = \vep$, so $n_0(D_y) = n_0(D_x) = 0$. 
	
	The proof of the theorem relies on the following claim:
	
	\noindent {\em Claim.} There exists a constant $N_1\in\NN$, depending only on $f, \vep, \cF$ and $\widetilde \cF$, such that 
	$$
	|n_0(I_w) - n_0(I_x)|\le N_1
	$$
	for all $w\in \widetilde \cF_\delta(x).$
	
	Once this claim is proven, one can immediately obtain that 
	\begin{align*}
		\frac{(H_{y,x})_*(\nu_y)(I_x)}{\nu_x(I_x)} &= \frac{\nu_y(H_{y,x}^{-1}(I_x))}{\nu_x(I_x)}\\
		&= \frac{\nu_y(I_y)}{\nu_x(I_x)}\\
		&\asymp_{L_H} 1,
	\end{align*}
	where $L_H = L^2\cdot e^{N_1H_\cF}$. Since $I_y$ is arbitrary, we get that $	\frac{d(H_{y,x})_*\nu_y}{d\nu_x}\asymp_{L_H}1$ and the proof of the theorem is complete.
	
	It only remains to prove the claim.
	\begin{proof}[Proof of the claim]
		Without loss of generality we assume that $x\in I_x$; otherwise replace $x$ by any $x'\in I_x$ and $y$ by $y'= [x', y]$ and note that $H_{y,x}$ can be identified with $H_{y',x'}$ (shrink the domain if necessary).

		Define 
		$$
		\tilde n_0 = \inf_{w\in\widetilde\cF_\delta(x)} n_0(I_w).
		$$
		In other words, $\tilde n_0$ is the first time that some $I_w$ reaches length $\vep = 2\delta$ under iteration. In particular we have $f^{\tilde n_0}(I_w)< K_f\cdot \vep$ for all $w\in\widetilde \cF_{\delta}(x)$ where $K_f$ is the norm of $Df$ restricted to $T_\cF$ (recall \eqref{e.exp}).
		
		To simplify notation we write $\tilde I_w = f^{\tilde n_0}(I_w)$ and note that since $\delta<\delta_1$, all $\tilde I_w$ are uncentered $\cF$-disks containing $f^{\tilde n_0}(w)\in\widetilde\cF_R(f^{\tilde n_0}(x))$ since the center is not expanding. We claim that $|\tilde I_w|_\cF\ge \iota$ for some $\iota>0$ depending only on $\vep,f,\cF$ and $\widetilde \cF$; the key here is that $\iota$ does not depend on $x,y$, $I_y$ or $\tilde n_0$. Once this is proven, one takes
		\begin{equation}\label{e.N1}
		N_1 = \left\lceil\frac{\log\vep-\log\iota}{\log\lambda}\right\rceil
		\end{equation}
		
		where $\lambda>1$ is given by the mini-norm of $Df$ on $T\cF$ (recall \eqref{e.exp}). This is the maximum time it takes for an $\cF$-disk of size $\iota$ to become at least $\vep$. Then we have that 
		$$
		n_0(I_w)\le n_0(x) + N_1
		$$
		for all $w\in\widetilde\cF_\delta(x)$. After switching $x$ and $w$ we get that 
		$$
		|n_0(I_w)- n_0(x)| \le N_1
		$$
		and the claim follows.
		
		It only remains to show that existence of $\iota$. For this purpose we fix $w_0\in\widetilde \cF_\delta(x)$ (take  closure if necessary) for which $\tilde n_0 = n_0(I_{w_0})$ and define the {\em semi-global holonomy $H_{\tilde I_{w_0},\tilde I_w}: \tilde I_{w_0}\to\tilde I_w$} as follows:
		$$
		H_{\tilde I_{w_0},\tilde I_w}(z) =  f^{\tilde n_0} \circ H_{w_0,y}^{loc} \circ f^{-\tilde n_0}(z)
		$$
		Since $H^{loc}_{w_0,w}(I_{w_0}) = I_w$ we get that $H_{\tilde I_{w_0},\tilde I_w}$ maps $\tilde I_{w_0}$ homeomorphically to $\tilde I_w$ for all $w\in\widetilde\cF_\delta(x)$. Also note that $H_{\tilde I_{w_0},\tilde I_w}$ can be also obtained by gluing together general holonomy maps of small domains in the sense of \cite[Definition (4)]{Plante}. Here a general holonomy $\widetilde H_{x,y}$ is defined for all $x,y$ on the same leaf of $\widetilde \cF$ whose existence in a small $\cF$-neighborhood of $x$ is due to transversality and continuity of the foliations, although the sizes of the domain and the image will depend on $x,y$.  In our  case, $\tilde w_0:= f^{\tilde n_0}(w_0)$ and $\tilde w:=f^{\tilde n_0}(w)$ are both contained in $\widetilde \cF_R(f^{\tilde n_0}(x))$, so $H_{\tilde I_{w_0},\tilde I_w}$ can be treated as extending the domain of $\widetilde H_{\tilde w_0,\tilde w}$ to $\tilde I_{w_0}.$
		
		Therefore the existence of $\iota$ is reduced to the following lemma:
		
		\begin{lemma}\label{l.iota}
			There exists $\iota,\vep_0>0$ such that for all $x\in\bM$ and $w\in \widetilde \cF_R(x)$, the general holonomy $\widetilde H_{x,w}: \cF_{\vep_0}(x)\to \cF(w)$ is well-defined, and its image contains $\cF_{\iota}(w).$
		\end{lemma}
		\begin{proof}
			First assume that there is no $\vep_0$ for which $\widetilde H_{x,w}$ is well-defined for all $w\in \widetilde\cF_R(x)$. Taking $\vep_0$ along the sequence $\frac1n$ we obtain sequences $(x_n), (w_n)\subset \bM$ with $w_n\in \widetilde \cF_R(x)$. Taking subsequence if necessary, we have $x_n\to x$ and $w_n\to w\in \widetilde \cF_{R+1}(x)$. Each $ \widetilde H_{x_n,w_n}: \cF(x_n)\to \cF(w_n)$ is only well-defined in $\cF_{\frac1n}(x_n)$. However, there exists $\vep_{x,w}>0$ for which $\widetilde H_{x,w}:\cF_{\vep_{x,w}}(x)\to\cF(w)$ is well-defined. Also keep in mind that $\cF_1(x_n)\to\cF_1(x)$ and $\cF_1(w_n)\to \cF_1(w)$ as $n\to\infty$. By transversality, continuity and the local product structure at both $x$ and $w$, for $n$ sufficiently large one can extend the domain of $\widetilde H_{x_n,w_n}$ to at least $\cF_{\frac{\delta_{x,w}}{2}}(x_n)$, which is a contradiction.
			
			The existence of $\iota>0$ is similar: assume that it does not exist, one obtains a sequence of points $x_n\to x$, $w_n\to w\in \widetilde \cF_{R+1}(x)$ for which $\widetilde H_{x_n,w_n}$ is defined on $\cF_{\vep_0}(x_n)$ but its image has only size $\frac1n.$ This means $\widetilde H_{x,w}$ must collapse $\cF_{\vep_0}(x)$ to the singleton $\{w\}$, which is impossible since $\widetilde H_{x,w}$ is a homeomorphism. The proof of the lemma is complete.
		\end{proof}
		
		Now that Lemma \ref{l.iota} is proven, we obtain that $\tilde I_{w} = \cR(H_{\tilde I_{w_0}, \tilde I_{w}})\supset \cR(\widetilde H_{\tilde w_0,\tilde w})$ which has size at least $\iota$; here $\cR(g)$ denotes the image of the map $g$. Then $N_1$ defined by \eqref{e.N1} satisfies the conclusion of the claim.
		
	\end{proof} 
	With the claim proven, the proof of Theorem \ref{t.holonomy} is complete.	
\end{proof}

\begin{remark}
	In the proof of Lemma \ref{l.iota}, $w_n$ being in $\widetilde \cF_R(x_n)$ for a uniform $R>0$ plays a crucial role; otherwise the limit $w$ may not be even in $\widetilde \cF(x)$.
\end{remark}

\section{Applications: Anosov diffeomorphisms on $\TT^n$ and partially hyperbolic diffeomorphisms on 3-nilmanifolds}\label{s.5}
In this section we provide several examples where the unstable foliation has the homogeneous growth property. As a result, Theorem \ref{t.properties},  \ref{t.1}, and \ref{t.2} all apply. 
\subsection{Anosov diffeomorphisms on $\TT^n$ with one-dimensional unstable bundle}\label{ss.5.1}
We start with a simple example. Let $f: \TT^n\to \TT^n$ be an Anosov diffeomorphism with $\dim E^u = 1$. Also denote by $A_f$ its linear part. Then, there is a topological conjugacy $h:\TT^n\to \TT^n$ with $h\circ f = A \circ h$, and $h$ maps each unstable leaf of $f$ to a linear unstable leaf of $A_f.$ See for instance \cite{Franks} and \cite{Manning}. In this case, it is well-known that $h$ is quasi-isometric: there exist constants $a>1, b>0$ such that for all $x,y$ on the same unstable leaf of $f$, one has
\begin{equation}\label{e.QI}
\frac1a \cdot d^u_A(hx,hy) - b \le d^u_f(x,y) \le a\cdot d^u_A(hx,hy) + b,
\end{equation}
where $d^u_f,d^u_A$ are the unstable leaf metrics of $f$ and $A$ respectively. This immediately leads to 
$$
	\frac1a\cdot |A^n(h(D))|_{\cF^u_A} - b\le |f^n(D)|_{\cF^u_f}\le a\cdot |A^n(h(D))|_{\cF^u_A} + b.
$$

Furthermore, we have that $|A^n(h(D))|_{\cF^u_A} = e^{nh_{top}(A)} |D|_{\cF^u_A} = e^{nh_{top}(f)} |D|_{\cF^u_A}$ since $A$ is linear, and $e^{h_{top}(A)}$ is the unstable eigenvalue of $A$. Therefore we have the following theorem:
\begin{theorem}\label{t.Anosov}
	Let $f: \TT^n\to \TT^n$ be a $C^1$ Anosov diffeomorphism with $\dim E^u = 1$. Then, the unstable foliation $\cF^u_f$ has homogeneous growth with $H_\cF = h_{top}(f)$.
	
	In this case, if $f$ is $C^{1+\alpha}$ and its unique measure of maximal entropy is also a Gibbs $u$-state, then the log Jacobian of $f$ on $\cF^u_f$ is cohomologous to a constant. {In particular, at every periodic point $p$ with period $n$, we have 
	$$
	|D_pf^n(v)| = e^{n h_{top}(f)}|v|
	$$
	for any vector $v\in E^u(p).$
	}
\end{theorem} 

\begin{proof}
	Let $\tilde \cF^u_f$ and $\tilde \cF^u_A$
	be the lifting unstable foliation
	of $f$ and $A$ on $\RR^d$, respectively. Let $\tilde x$, $\tilde y$ be lifts of $x$, $y$ where $\tilde y \in \tilde\cF^u_f(\tilde x)$, and
	$H : \RR^d \to \RR^d$ be the lift of conjugacy $h$. Then we have
	$$
	d^u_A(hx, hy) = d_{\tilde\cF^u_A}(H\tilde x,H\tilde y) = d_{\RR^d}(H\tilde x,H\tilde y) \mbox{ and } d^u_f(x, y) = d_{\tilde \cF^u_f}(\tilde x, \tilde y).
	$$
	Since the foliation $\tilde \cF^u_f$ is quasi-isometric on $\RR^d$ (\cite[Lemma 2.2]{GS}), there exists $a > 1, b_1 > 0$ such that
	$$
	\frac1a\cdot d_{\RR^d}(\tilde x,\tilde y)-b_1\le d_{\cF^u_f}(\tilde x,\tilde y)\le a\cdot d_{\RR^d}(\tilde x,\tilde y)+b_1.
	$$
	Moreover, we have $\|H - Id_{\RR^d}\| < C_0$, which implies
	$$
		\frac1a\cdot d_{\RR^d}(H\tilde x,H\tilde y)-\left(\frac{2C_0}{a}+b_1\right)\le d_{\cF^u_f(\tilde x,\tilde y)}\le a\cdot d_{\RR^d}(H\tilde x,H\tilde y)+ (2aC_0+b_1).
	$$
	By taking $b = 2aC_0 + b_1$, we finish the proof.
\end{proof}

\subsection{Partially hyperbolic diffeomorphisms on 3-nilmanifolds} Now
we move on to partially hyperbolic diffeomorphisms on 3-nilmanifolds. We
consider two cases:

\subsubsection{(Partially) hyperbolic diffeomorphisms on $\TT^3$} Ma\~n\'e \cite{Ma} constructed one of the first robustly transitive, yet not uniformly hyperbolic diffeomorphism on $\TT^3$. It is obtained via a $C^0$ small perturbation at a fixed point of an Anosov diffeomorphism on $\TT^3$. His construction was later generalized to derived from Anosov (DA) diffeomorphisms: they are partially hyperbolic diffeomorphisms on $\TT^3$ that are isotopic to linear Anosov automorphisms. See for instance \cite{FPS} for a more detailed discussions.

We let $A$ be a linear Anosov automorphism on $\TT^3$ with eigenvalues $0 < \kappa_1<\kappa_2< 1 < \kappa_3$, and $f$ be a partially hyperbolic diffeomorphism in the isotopy class of $A$. In this case, $f$ has dominated splitting $E^{cs}\oplus E^u$ with $\dim E^{cs} = 2$, and $E^u$ integrates into a one-dimensional foliation $\cF^u$. 

It has been proven by Franks \cite{Franks} that there is a semi-conjugacy $h:\TT^3\to\TT^3$ with $h\circ f = A\circ h$. Note that $h$ maps each leaf of $\cF^u$ homeomorphically to a leaf of $\cF^u_A$ (see \cite[Corollary 7.7]{Potrie}). Moreover, the following has been proven in 
\cite{HP}.
\begin{proposition}\label{p.1}
	The unstable foliation $\cF^u$ of $f$ is quasi-isometric: there exist constants $a,b>0$ such that for any points $x,y$ on the same leaf of $\cF^u$, we have 
	$$
	d_{\tilde \cF^u}(\tilde x,\tilde y)\le a\cdot d_{\RR^3}(\tilde x,\tilde y) + b
	$$
	where $\tilde x, \tilde y$ are lifts of $x,y$ to the universal covering $\RR^3$, and $\tilde \cF^u$ is the lift of $\cF^u$.
\end{proposition}
Combining this with the fact that the lift of Franks' semi-conjugacy $\tilde h$ is at a uniformly bounded distance to the identity map of $\RR^3$ (see for instance \cite[Section 10]{HP18}) and following the same proof as the Anosov case in Section \ref{ss.5.1}, we obtain:
\begin{theorem}\label{t.DA}
	$\cF^u$ has homogeneous growth, with $H_\cF = h_{top}(A) = \log \kappa_3$.
	
	In addition, if $f$ is $C^{1+\alpha}$ and the unique MME of $f$ is also a Gibbs $u$-state, then the log Jacobian of $f$ on $\cF^u$ is cohomologous to a constant.  In particular, at every periodic point $p$ with period $n$, we have 
		$$
		|D_pf^n(v)| = \kappa_3^n\cdot|v|
		$$
		for any vector $v\in E^u(p).$
	
\end{theorem}
Here the second statement is due to the fact that the unique MME of $f$ is also the unique measure of maximal $\cF^u$-entropy. For more details, see \cite[Proposition 6.1]{UVYY1}.

\subsubsection{Partially hyperbolic diffeomorphisms on 3-nilmanifolds other than $\TT^3$}

Next we consider any partially hyperbolic diffeomorphism $f$ on a 3-nilmanifold $\bM$ other than $\TT^3$. These manifolds are quotients of the 3-Heisenberg group, and can be viewed as non-trivial circle bundles over $\TT^2$. These classes of diffeomorphisms have been studied in \cite{UVY} and \cite{RHRHTU}.

Denote by $E^s\oplus E^c\oplus E^u$ the dominated splitting of $f$ into one-dimensional bundles, and let $\cF^u$ be the one-dimensional unstable foliation. It has been shown in \cite[Section 8]{UVYY1} (see also \cite[Corollary 1.2]{H13}) that there exists a continuous surjective map $h: \bM\to\TT^2$ with the following properties:
\begin{enumerate}
	\item $ h\circ f = A \circ h$ where $A$ is some Anosov automorphisms of $\TT^2$; in other words, $h$ is a semi-conjugacy.
	\item $h$ maps each unstable leaf of $f$ homeomorphically to a linear unstable leaf of $A$.	
\end{enumerate}

We claim:
\begin{proposition}
	$h$ is a quasi-isometry on each unstable leaf of $h$ in the sense of \eqref{e.QI}.
\end{proposition}
\begin{proof}
	Denote by $\mathbb H$ the 3-Heisenberg group consisting of matrices of the form $$H(x,y,z):=\begin{pmatrix}
		1 & x & z\\
		0 & 1 & y\\
		0 & 0 & 1
	\end{pmatrix}$$
	and by $\Pi: \mathbb H\to \RR^2$ the group homomorphism $\Pi(H(x,y,z)) = (x,y).$ Then, $\mathbb H$ is the universal covering of $\bM$.
%
%

	The semi-conjugacy $h: \bM\to \TT^2$ lifts naturally to some $\tilde h:\mathbb H \to \RR^2$ in the following way: first find a (center) leaf conjugacy $h_1: \bM\to\bM$ that leaf conjugates $f$ to a nil-automorphism $B$ (the existence is given by \cite{H13}), then lift $h_1$ to some $\tilde h_1:\mathbb H\to\mathbb H$ that center leaf conjugates $\tilde f$ and $\tilde B$. Then, $\tilde h = \Pi \circ \tilde h_1$ semi conjugates $\tilde f$ and $\tilde A:\RR^2\to\RR^2$. Observe that:
	\begin{enumerate}
		\item $\tilde h$ sends each leaf of $\tilde \cF_f^u$ homeomorphically to a linear leaf of $\tilde A.$
		\item $\tilde h$ is at uniformly bounded $C^0$-distance to $\Pi$ because $\tilde h_1$ is at a  bounded distance to the identity map on $\mathbb H$.
		\item  $\Pi$ is bi-Lipschitz with uniformly bounded Lipschitz norm when restricted to each unstable leaf of $\tilde B$, since those leaves have uniform angle with the center leaves of $\tilde B$ which are $A(x,y,tz)$.
	\end{enumerate}  
	Combining these facts, we see that $\tilde h$ is a quasi-isometry, and so is $h$.
\end{proof}

\begin{remark}
	Here the proof is in fact a circular reasoning. In \cite{H13} it was first proven that $\cF^u_f$ is quasi isometric, and therefore $\tilde f$ has a global product structure. The semi-conjugacy $h:\bM\to\TT^2$ was built using this fact. 
\end{remark}

The following theorem follows immediately. 
\begin{theorem}\label{t.5.4}
	Let $f$ be any $C^1$ partially hyperbolic diffeomorphism on a $3$-dimensional nilmanifold other than $\TT^3$. Then $\cF^u_f$ has homogeneous growth, with $H_\cF = h_{top}(A) = h_{top}(f)$.
	
	Furthermore, if $f$ is $C^{1+\alpha}$ and the unique measure of maximal $\cF^u$-entropy of $f$ is also a Gibbs $u$-state, then the log Jacobian of $f$ on $\cF^u$ is cohomologous to a constant via a measurable solution $\phi$. {In particular, at every periodic point $p$ with period $n$, we have 
		$$
		|D_pf^n(v)| = e^{nh_{top}(f)}|v|
		$$
		for any vector $v\in E^u(p).$
	} 
\end{theorem}
Here we remark that the uniqueness of the measure of maximal $\cF^u$-entropy is proven in \cite{UVYY1}. This measure is also a measure of maximal entropy, but the latter may not be unique. 

{
Indeed it was proven in \cite{RHRHTU} that a dichotomy holds: either there exists a unique measure of maximal entropy with zero center exponent, or there exists two ergodic measures of maximal entropy, denoted by $\mu^\pm$. The center exponent of $\mu^-$ (resp. $\mu^+$) is negative (resp. positive), and it was shown in \cite{UVYY1} that it is the unique measure of maximal $\cF^u$-entropy (resp. the unique measure of maximal $\cF^s$-entropy by considering $f^{-1}$). 
} 

Using \cite[Theorem A]{Wilkinson}, we can improve the measurable solution $\phi$ to a continuous solution provided that $f$ is volume preserving.
%

\begin{theorem}
	Assume that $f$ is a $C^{1+\alpha}$ volume preserving partially hyperbolic
	diffeomorphism on a 3-dimensional nilmanifold other than $\TT^3$. If the
	measure of maximal entropy of $f$ coincide with Lebesgue volume
	$$\mu_{MME} = m,$$
	then
	\begin{itemize}
		\item  if $\lambda^c(m) \ge 0$, then there exists a continuous function $\tilde\phi_s$ satisfying
		$$
		\log |D_xf|_{E^s(x)}| = \tilde\phi_s\circ f(x) - \tilde\phi_s + h_{top}(f)$$
		\item if $\lambda^c(m) \le  0$, then there exists a continuous function $\tilde\phi_u$ satisfying
		$$
		\log |D_xf|_{E^u(x)}| = \tilde\phi_u\circ f(x) - \tilde\phi_u + h_{top}(f)
		$$
		\item if $\lambda^c(m) = 0$, then $f$ is smoothly conjugate to a rotation extension
		of the Anosov automorphism $A$ on $\TT^2$.	
	\end{itemize}
\end{theorem}

\begin{remark}
	Even though the solution $\phi$ given by Theorem \ref{t.5.4} is only $\mu$-measurable rather than Borel measurable, one can directly check that \cite[Proposition 4.7]{Wilkinson} still applies, and therefore \cite[Theorem A]{Wilkinson} is applicable.
\end{remark}

\section{Applications: perturbations of time-one maps of geodesic flows}\label{s.geodesic}
Let $S$ be a compact surface with negative curvature. Write $\bM = T^1S$ for the unit tangent bundle of $S$, and $(g_t)_{t\in\RR}$ for the geodesic flow on $\bM$. It is well-known that $(g_t)_{t\in\RR}$ has a hyperbolic splitting $E^s\oplus \langle X\rangle \oplus E^u$. $g_1$, the time-one map, is naturally a partially hyperbolic diffeomorphism on $\bM,$ and we will consider its perturbation in the space of $C^1$ diffeomorphisms on $\bM$. Let $\cU$ be a small $C^1$ neighborhood of $g_1$, then every $f\in \cU$ is also partially hyperbolic. Furthermore, the unstable bundle of $f$ integrates into a foliation which we denote by $\cF^u(f)$.

The next proposition gathers some useful properties concerning diffeomorphisms in $\cU.$

\begin{proposition}\label{p.p}
	Let $g_1: \bM\to\bM$ be the time-one map of the geodesic flow on a compact surface of negative curvature, and $\cU$ be a $C^1$ open neighborhood of $g_1$ in $\mbox{Diff\ }^1(\bM)$. Then, for $\cU$ sufficiently small, the following hold for every $f\in \cU:$
	\begin{enumerate}
		\item (\cite[Theorem 7.1]{HPS}) $f$ is dynamically coherent, and is center leaf conjugated to $g_1$ via a homeomorphism $h_f$ that is $C^0$ close to the identity.
		\item (\cite[Lemma 2.5]{SY}) $f$ is not expanding on the center leaves in the sense that there exist $K_1,K_2>0$ such that for any $c\in\NN$, if $y\in \cF^c(x,f)$ with $d^c(x,y)<c K_1$ then 
		$$
		d^c(f^n(x),f^n(y))\le c K_2, \forall n\in\NN.
		$$ 
		Here $d^c$ is the distance on center leaves. 
		\item (\cite[Proposition 2.7]{SY}) The $c$-holonomy is globally defined within $cu$-leaves: if $y\in \cF^{c}(x,f)$ then the holonomy $h^{c}_{x,y}: \cF^{cu}(x,f)\to \cF^{cu}(y,f)$ is well-defined and continuous. 
	\end{enumerate}
\end{proposition}

\subsection{Homogeneous growth is robust near $g_1$}
The goal of this subsection is to prove:
\begin{theorem}\label{t.geodesic}
	Let $g_1: \bM\to\bM$ be the time-one map of the geodesic flow on a compact surface of negative curvature. Then, there exists a $C^1$ neighborhood $\cU$ of $g_1$, such that for every $f\in \cU$, $\cF^u(f)$ has homogeneous exponential growth with $H_\cF = h_{top}(f)$.  
\end{theorem}

The rest of this section is dedicated to the proof of this theorem. The proof requires the following lemmas.
 {
\begin{lemma}\label{l.holonomy}
	For $\cU$ sufficiently small, all $f\in \cU$ satisfies that the foliation $\cF^{cs}$ is a transversal foliation in the sense of Definition \ref{d.TF}, along which $f$ is not expanding in the sense of Definition \ref{d.NE}. Furthermore, the constant $R$ can be made uniform in $\cU$. 
\end{lemma} 

\begin{proof}
	This lemma follows directly from compactness and the non-expanding property of $f$ on the center leaves.
\end{proof} 

Similar to Section \ref{s.holonomy}, we denote by $\cF^*_r(x,f)$, $* = cs,u$ the (closed) ball of radius $r$ centered at $x$ in the leaf $\cF^*(x,f)$.
Henceforth, when no confusion is caused, we will call $\cF^u(f)$-disks simply by $u$-disks.

Let $\delta>0$ be the size of the local product structure formed by local leaves of $\cF^{cs}$ and $\cF^u$, and keep in mind Remark \ref{r.vep}. By the continuity of invariant foliations, $\delta$ can be made uniform in a $\cU$. As in Section \ref{s.holonomy}, $H^{loc}_{y,z}:\cF^u_\delta(y,f)\to \cF^u_\delta(z,f)$ are well-defined for all $y,z\in \cF^{cs}_\delta(x)$ and all $x\in\bM.$
We define the product neighborhood of a point $x\in\bM$ at scale $\delta$ to be the set 
$$
\cN_\delta(x,f) = \bigcup_{y\in \cF^u_{\delta}(x,f)} \cF^{cs}_\delta(y,f).
$$
For each $y\in \cN_\delta(x,f)$ we write $D^{x}_y$ the connected component\footnote{Note that if $\delta$ is sufficiently small, then there is only one connected component.} of $\cF^u_{4\delta}(y,f)\cap \cN_\delta(x,f)$ that contains $y$.

For the geodesic flow $(g_t)_{t\in\RR}$, it is well-known that the unstable foliation $\cF^u(g)$ is minimal in the sense that every leaf is dense in $\bM$. By compactness, there exists $N_0\in\NN$, such that for every $u$-disks $D$ with $|D|_{u} \ge \delta$, we have $g_{N_0}(D)$ fully crosses $\cN_\delta(x,g_1)$ for every $x\in\bM$.
It is straightforward to verify that this property is $C^1$  robust in the following sense: there exists a small $C^1$ neighborhood $\cU$ of $g_1$ such that for every $f\in\cU$ and every $u$-disks $D$ with $|D|\ge \delta$, $f^{N_0}(D)$ must fully cross $\cN_\delta(x,f)$. Then, the following lemma will allow us to match the sizes of different $u$-disks under iteration.

%
%
%
%

\begin{lemma}\label{l.compare2}
	Let $K_f = \|f\|_{C^1}$. Then, there exists a constant $C_{1}>1$, such that for all $u$-disks $D_1$ and $D_2$ with $|D_i|_u\in [\delta,K_f\delta)$, one has
	\begin{equation}\label{e.compare}
		|f^{n}(D_1)|_u\asymp_{C_{1}}|f^{n}(D_1)|_u
	\end{equation}
	for all $n\in\NN$
\end{lemma}
\begin{proof}
	We will only consider the case where $|D_1|_u = |D_2|_u = \delta$; once this is proven, the general case can be obtained by covering each $D_i$ by no more than $\lceil K_f + 1\rceil$ many pieces, each of which has length one. 
	
	Let $D_1,D_2$ be any $u$-disks with $|D_i|_u=\delta$, $i=1,2$ and let $x_0$ be the center of $D_2$. By the discussion above, we see that $f^{N_0}(D_1)$ fully crosses $\cN_\delta(x_0,f)$. Denote by $\widetilde D_1$ the $u$-disk that is $f^{N_0}(D_1)\cap \cN_\delta(x_0,f)$,  and write $y_0\in\widetilde D_1$ the transversal point of intersection $\cF^{cs}_\delta(x_0)\pitchfork \widetilde D_1$. Below we will show that there exists a constant $C_0>1$, independent of $D_1,D_2$, such that 
	\begin{equation}\label{e.compare1}
	|f^{n}(\widetilde D_1)|_u \ge \frac{1}{C_0}|f^n(D_2)|_u, \mbox{ for all $n\in\NN$}.
	\end{equation}
	This immediate implies that 
	$$
	|f^{n+N_0}( D_1)|_u \ge|f^{n}(\widetilde D_1)|_u \ge \frac{1}{C_0}|f^n(D_2)|_u \ge  \frac{1}{K_f^{N_0}C_0}|f^{n+N_0}(D_2)|_u
	$$
	for all $n\in \NN$. Then the lemma follows by taking $C_1\gg\lambda^{N_0}C_0$ to cover the first $N_0$ iterates and switching $D_1$ with $D_2$. 
	
	Below we prove \eqref{e.compare1}.

	Recall that $\iota,\vep_0>0$ are the constants given by Lemma \ref{l.iota}.
	For any $n\in\NN$ we write $L_n = \lceil{|f^n(D_2)|_u}/{\vep_0}\rceil+ 1$, and cut $f^n(D_2)$ into $L_n$-many $u$-disks with lengths $\vep_0$ (except maybe for the last one) that only intersect at boundary points. These disks can be pulled back by $f^{-n}$ to become a finite partition of $D_2$ into $L_n$-many $u$-disks $D_2^1,\cdots, D_2^{L_n}$ with boundary points $x_1,\cdots, x_{L_n+1}$; each $D_2^j$ is sent by $H^{loc}_{x_0,y_0}$ to a disk $\widetilde D_1^j = H^{loc}_{x_0,y_0}(D_2^j)\subset\widetilde D_1$. $\{\widetilde D_1^j, j=1,\cdots, L_n\}$ form a finite partition of $\widetilde D_1$ into $u$-disks, all of which only intersect at their boundary points $y_j = H^{loc}_{x_0,y_0}(x_j)$. 
	
	Iterating by $f^n$, we see that the general holonomy $\widetilde H_{f^n(x_j),f^n(y_j)}$ must send each $f^{n}(D_2^j)$ to $f^n(\widetilde D_1^j)$. Since $f$ is not expanding on $\cF^{cs}(f)$ and $|f^{n}(D_2^j)|_u=\vep_0$ for all $j$ due to the construction, Lemma \ref{l.iota} shows that $|f^n(\widetilde D_i^j)|_u\ge \iota$ for each $j$. Summing over $j$, we obtain that 
	$$
	|f^n(\widetilde D_1)|\ge L_n\cdot\iota \ge \frac{\iota}{\vep_0}|f^n(D_2)|,
	$$
	as required. This finishes the proof of \eqref{e.compare1} and consequent that of Lemma \ref{l.compare2}.

\end{proof}

\begin{remark}
	We note that all the constants $\delta, \iota,\vep_0, R,K_f$ can be made uniform in $\cU$.
\end{remark}

}

Given a $u$-disk $D$, following the definition of unstable topological entropy in \cite{HSX}, we define 
$$
\chi_u(D) =\limsup_{n\to\infty}\frac1n\log|f^n(D)|.
$$ 
Then, Lemma \ref{l.compare2} implies that $\chi_u(D)$ is independent of $D$. 
Denote by $H_f$ the common value of $\chi_u(D)$. Then, by \cite{HSX} and \cite[Theorem C]{HHW}, we have
\begin{equation}\label{e.uentropy}
	H_\cF = \chi_u(D) = h_{top}^u(f).
\end{equation}

So far we have shown that the (limsup of) exponential growth rate of $|f^n(D)|_u$ are all equal to $H_\cF = h_{top}^u(f)$. It remains to obtain an asymptotic bound of the form $|f^n(D)|_u\asymp e^{nH_\cF}$. For this we use ideas similar to \cite[Section 6]{CPZ} and \cite[Section 6]{CPZ1}. The proof in our case is much simpler since we only need to estimate the length, whereas \cite{CPZ,CPZ1} consider the cardinality of $u$-separated sets.

We define
\begin{equation}\label{e.Z_n}
	Z_n :=\sup_{|D|_u=\delta} |f^n(D)|_u,
\end{equation}
where the supremum is taken over all $u$-disks $D$ with $|D|_u = \delta$.  The following lemma establishes a version of sub-multiplicativity for $Z_n.$

\begin{lemma}\label{l.submulti}
	For every $k,\ell\in\NN$ and any $u$-disk $D$ with $|D|_u\in [\delta,K_f\delta)$, we have 
	$$
	|f^{k+\ell}(D)|_u\le \frac{2}{\delta}Z_\ell \cdot  |f^k(D)|_u.
	$$
\end{lemma}
\begin{proof}
	Let $D$ be any $u$-disk with $|D_u|\in [\delta,K_f\delta)$. 
	We cut $f^k(D)$ into $n_D(k):=\left(\lfloor |f^k(D)|_u/\delta\rfloor + 1\right)  $-many piece of $u$-disks with length $\delta$ (except for the last piece whose length is at most $\delta$), denoted by $D_1,\cdots, D_{n_D(k)}$. Then, we have
	\begin{align*}
		|f^{k+\ell}(D)|_u &= \sum_{j=1}^{n_D(k)}|f^\ell(D_j)|_u \le \sum_{j=1}^{n_D(k)} Z_\ell \\
		&\le \left(\frac{|f^k(D)|_u}{\delta} + 1\right) Z_\ell\\
		&\le \frac{2}{\delta}Z_\ell\cdot |f^k(D)|_u.
	\end{align*}
\end{proof}

The next proposition establishes the lower bound on the growth of $|f^n(D)|_u$.
\begin{proposition}\label{p.low}
	There exists a constant $C_2>0$ such that every $u$-disk $D$ with $|D|_u\in[1,K_f)$ satisfies 
	$$|f^n(D)|_u\ge C_2e^{nH_\cF}, \forall n\in\NN.$$
\end{proposition}
\begin{proof}
	Let $D$ be any $u$-disk with $|D|_u=1$. We apply Lemma \ref{l.submulti} iteratively to obtain
	$$
	|f^{nk}(D)|_u \le \frac{2}{\delta}Z_k |f^{(k-1)n}(D)|_u\le \cdots \le\left(\frac{ 2}{\delta}\right)^{k} (Z_n)^k.  
	$$
	Taking logarithm, dividing by $nk$ and sending $k$ to infinite\footnote{Here the limit of the left-hand side exists and equals $H_\cF$ is due to \cite[Lemma 6.2 and 6.3]{CPZ}.} yields
	$$
	H_\cF \le \frac{\log 2-\log\delta}{n} + \frac{\log Z_n}{n},
	$$
	which becomes 
	$$
	\log Z_n\ge -\log 2 +\log\delta+ nH_\cF,  
	$$
	that is, $Z_n\ge \frac\delta2e^{nH_\cF}$. It follows from Lemma \ref{l.compare2} that any $u$-disk $D$ with $|D|_u=\delta$ satisfies 
	$$
	|f^n(D)|_u\ge \frac{1}{C_{1}} Z_n\ge \frac{\delta}{2C_{1}}\cdot e^{nH_\cF},
	$$
	that is, the lemma holds with $C_2 = \delta(2C_{1})^{-1}$.
	For the general case $|D|_u\in[\delta,K_f\delta)$, the same estimate follows by considering a sub-disk with length $\delta$. 
\end{proof}

To obtain the upper bound, we establish the following form of super-multiplicativity for $Z_n.$
\begin{lemma}\label{l.up}
	For any $u$-disk $D$ with $|D|_u=\delta$ and any $k,\ell\in\NN$, one has 
	$$
	|f^{k+\ell}(D)|_u \ge \frac{1}{2\delta C_{1}}Z_\ell\cdot |f^k(D)|_u. 
	$$ 
\end{lemma}
\begin{proof}
	Similar to Lemma \ref{l.submulti}, we cut $|f^k(D)|_u$ into $n'_D(k):=\lfloor |f^k(D)|_u/\delta \rfloor $ many piece of $u$-disks with length $\delta$ (except for the last one, whose length is in $[\delta,2\delta)$). Then, Lemma \ref{l.compare2} gives 
	\begin{align*}
	|f^{k+\ell}(D)|_u &= \sum_{j=1}^{n'_D(k)}|f^\ell(D_j)|_u \ge \frac{1}{C_{1}} \sum_{j=1}^{n'_D(k)} Z_\ell \\
	&\ge \frac{1}{2\delta C_{1}} Z_\ell\cdot  |f^k(D)|_u.
\end{align*}
\end{proof}
Then, the next proposition provides an upper bound for the growth of $|f^n(D)|$.
\begin{proposition}\label{p.up}
	There exists a constant $C_3>1$ such that every $u$-disk $D$ with $|D|_u\in[1,K_f)$ satisfies 
	$$|f^n(D)|_u\le C_3e^{nH_\cF}, \forall n\in\NN.$$
\end{proposition}
\begin{proof}
	We use Lemma \ref{l.up} iteratively to obtain, for any $u$-disk $D$ with $|D|_u = \delta$,
	\begin{align*}
		|f^{kn}(D)|_u\ge \frac{1}{2\delta C_{1}}Z_n \cdot |f^{(k-1)n}D|_u \ge \cdots \ge \frac{1}{(2\delta C_{1})^k}\cdot Z_n^k.
	\end{align*}
	Taking logarithm, dividing by $kn$ and sending $k$ to infinite,\footnote{Here the limit of the left-hand side exists and equals $H_\cF$ is due to \cite[Lemma 6.4]{CPZ}.} we see that 
	$$
 H_\cF \ge -\frac{\log 2\delta C_{1}}{n} + \frac{\log Z_n}{n}, 
	$$
	that is,
	$$
	Z_n\le 2\delta C_{1}e^{nH_\cF}.
	$$
	By the definition of $Z_n$ we conclude that $|f^n(D)|_u\le Z_n\le 2\delta C_{1}e^{nH_\cF}$ for any $u$-disk $D$ with $|D|_u = \delta$. For the general case of $|D|_u\in[\delta,K_f\delta)$, one can find a larger disk $\widetilde D$ such that $D\subset \widetilde D$ and $|\widetilde D|_u = \lceil K_f\rceil + 1$. Then, cutting $\widetilde D$ into $|\widetilde D|_u$-many piece of $u$-disks, each of which has length $\delta$, we get
	$$
	|f^n(D)|_u\le |f^n(\widetilde D)|_u\le |\widetilde D|_u\cdot 2\delta C_{1}e^{nH_\cF},
	$$
	and the proposition follows by taking $C_3 = |\widetilde D|_u\cdot 2\delta C_{1} = (\lceil K_f\rceil + 1)2\delta  C_{1}$.
\end{proof}

Combining Proposition \ref{p.low} and Proposition \ref{p.up} we see that $\cF^u(f)$ has homogeneous growth with constant $C_G = \max\{C_2^{-1}, C_3\}$ and $H_\cF = h^u_{top}(f)$. It remains to show that $h^u_{top}(f) = h_{top}(f)$. We introduce some notations first. Recall that $f$ has local product structure at scale $\delta>0$. We define, for $x\in\bM$,
$$
A(x,\delta) = \{y: \exists x^*\in \cF^*_\delta(x,f), *=c,u, \mbox{s.t. } \{y\} = \cF^c_{2\delta}(x^u,f) \cap \cF^u_{2\delta}(x^c,f) \},
$$
and consider the following product neighborhood of $x$:
$$
D(x,\delta) = \cF^s_{\delta}(A(x,\delta),f) = \bigcup_{z\in A(x,\delta)} \cF^s_\delta(z,f).
$$
For any compact set $K\subset \bM$, we denote by $h(f,K)$ the topological entropy of $f$ on $K$, that is, the exponential growth rate of the maximal cardinality of separated sets on $K$. Keep in mind that $f$ is partially hyperbolic with one-dimensional center, and is therefore entropy expansive (see \cite{LVY}). For this reason, the scale of the separated sets does not matter as long as it is sufficiently small. 

We recall the following result from \cite{SY}.

\begin{proposition}\label{p.sy}\cite[Proposition 2.15]{SY}
	Let $f$ be a dynamically coherent partially hyperbolic diffeomorphism with one-dimensional center on a compact manifold $\bM$. Suppose that $f$ is not expanding along the center foliation. Then, there exists $\delta_0>0$ such that for all $\delta<\delta_0$ one has 
	$$
	h(f,D(x,\delta)) = h(f,\cF^u_\delta(x,f)).
	$$
\end{proposition}

Next we prove that $h_{top}(f) = h_{top}^u(f)$ robustly near $g_1$. 
\begin{proposition}\label{p.entropy}
	Let $g_1: \bM\to\bM$ be the time-one map of the geodesic flow on a compact surface of negative curvature. Then, there exists a $C^1$ neighborhood $\cU$ of $g_1$, such that for every $f\in \cU$, one has 
	$$
	h_{top}(f) = h_{top}^u(f).
	$$
\end{proposition}

\begin{proof}
	Note that any $f$ sufficiently close to $g_1$ satisfies the assumptions of Proposition \ref{p.sy}. Fix any $\delta<\delta_0$ and cover $\bM$ by finitely many $D(x_i,\delta)$, $i=1,\cdots, N$. By Proposition \ref{p.sy} we have 
	$$
	h_{top}(f) = \max_{1\le i\le N} h(f, D(x_i,\delta)) = \max_{1\le i\le N} h(f,\cF^u_\delta(x_i,f)).
	$$
	
	On the other hand, by \cite[Section 4.3, Proof of Theorem C]{HHW} and \eqref{e.uentropy} we have 
	$$
	h(f,\cF^u_\delta(x_i,f)) = \chi_u(\cF^u_\delta(x_i,f)) = h^u_{top}(f)
	$$
	for every $i$.
	Combining these two equalities, we have $h_{top}(f) = h^u_{top}(f)$ as claimed. 
\end{proof}

\begin{proof}[Proof of Theorem \ref{t.geodesic}]
	By Proposition \ref{p.low} and Proposition \ref{p.up}, $\cF^u(f)$ has homogeneous exponential growth with $C_G = \max\{C_2^{-1}, C_3\}$ for every $f$ sufficiently close to $g_1$. Furthermore, by \eqref{e.uentropy} and Proposition \ref{p.entropy} we get 
	$$
	H_\cF= h^u_{top}(f) = h_{top}(f),
	$$ 
	and the proof of the theorem is complete. 
\end{proof}

\subsection{Rigidity for perturbations of time-one maps of geodesic flows}
In this section we establish some rigidity results for $C^{1+\alpha}$ diffeomorphisms $f\in\cU$. For this purpose, for any ergodic measure $\mu$ we denote by $\lambda^*(\mu),*=s,c,u$ the stable/center/unstable Lyapunov exponents of $\mu$. 
The next theorem concerns the rigidity of the volume measure from a Lyapunov exponents perspective.  This can be seen as the counterpart of \cite{SY19} for perturbations of geodesic flows.

\begin{theorem}\label{t.geodesic.rigidity}
	Let $g_1: \bM\to\bM$ be the time-one map of the geodesic flow on a compact surface of negative curvature. Then, there exists a $C^1$ neighborhood $\cU$ of $g_1$, such that for every $C^{\infty}$ diffeomorphism $f\in \cU$ that preserves the volume $m$, the following hold:
	
	Assume one of the following: 
	\begin{enumerate}
		\item Either $\lambda^s(m) = \lambda^u(m) = h_{top}(f)$.
		\item Or $\lambda^s(\mu_{{\operatorname{MME}}}) = \lambda^u(\mu_{{\operatorname{MME}}}) = h_{top}(f)$, where $\mu_{{\operatorname{MME}}}$ is a measure of maximal entropy of $f$.
	\end{enumerate}
	 Then $f$ is the time-one map of a smooth flow, and the log Jacobian of $f$ on both $\cF^u$ and $\cF^s$ are cohomologous to constants via continuous functions $\phi^*$, $*=s,u$.
\end{theorem}

\begin{proof}
	By \cite{KK}, $g_1$ is accessible and hence stably ergodic (see \cite{GPS} and \cite{BW}). In particular, if $f$ is $C^\infty$, volume preserving, and $C^1$ close to $g_1$, then the volume measure is ergodic for $f$.
	
	We first consider Assumption (1). Under the assumption that $f$ is volume preserving and $\lambda^u(m) = \lambda^s(m) = h_{top}(f) = h^u_{top}(f)$ where the last inequality is due to Theorem \ref{t.geodesic}, we see that $m$ has zero center exponent and is an ergodic $u$-MME that is also a Gibbs $u$-state. Then by Theorem \ref{t.2}, the log Jacobian of $f$ is cohomologous to a constant via a $m$-measurable solution $\phi$. Then, \cite[Theorem A]{Wilkinson} shows that the solution can be made continuous.

	Next we show that the $\cF^{cs}$ is absolutely continuous. Recall that the references measures $\Gamma(D)$ are defined on every $u$-disk $D$. Let $\Theta\gg1$ be given, and we define 
	\begin{align*}
		\Gamma_\Theta(D) = \{\mu: &\mu\mbox{ is a probability measure on D that is absolutely} \\&\mbox{ continuous w.r.t. some }\nu\in\Gamma(D)  \mbox{ with  $\frac{d\mu}{d\nu}\asymp_{\Theta} 1$}\}.		
	\end{align*}
	Since $m$ is a Gibbs $u$-state and also a $u$-MME, there exist $\Theta>1$ such that for $m$-a.e. $x\in \bM$, the conditional measures $m_x:= m_x^{\cF^u,\cA}\in \Gamma_\Theta(\cF^u_\cA(x))$ where $\cA$ is the partition given in Section \ref{s.3}.  
	On the other hand, Theorem \ref{t.properties} implies that there exists $\Theta'\gg\Theta$ such that,  if $(x_n)$ is a sequence of points converging to some $x\in\bM$, then for any $\nu_n\in \Gamma_\Theta(\cF^u_\cA(x_n))$, any weak-* limit point of $\nu$ of $(\nu_n)$ satisfies $\nu\in \Gamma_{\Theta'}(\cF_\cA(x))$.

	Now let $x\in\bM$ be arbitrary and $y\in \cF^{cs}_\delta(x)$ where $\delta$ is the size of the local product structure. We will show that the $cs$-holonomy $H^{loc}_{x,y}$ is absolutely continuous with respect to $m_x$ and $m_y$ with uniformly bounded Jacobian. For this purpose, take a sequence $(x_n)$ consisting of typical points of $m$ that converges to $x$, and similarly $(y_n)$ converges to $y$. Since $m_{x_n}\in\Gamma_\Theta(\cF_\cA(x_n))$ and $m_{y_n}\in\Gamma_\Theta(\cF_\cA(y_n))$, we see that $m_x\in \Gamma_{\Theta'}(\cF_\cA(x))$ and $m_y\in \Gamma_{\Theta'}(\cF_\cA(y))$, since the leaf-wise volumes are continuous. On the other hand, by Theorem \ref{t.holonomy}, $H^{loc}_{x,y}$ is absolutely continuous with respect to measures in $\Gamma(\cF_\cA(x_n))$ and $\Gamma(\cF_\cA(y_n))$ with uniformly bounded Jacobian. This means that the Jacobian is also uniformly bounded with respect to measures in $\Gamma_{\Theta'}(\cF_\cA(x))$ and $\Gamma_{\Theta'}(\cF_\cA(y))$ (although the bound depends on $\Theta'$). In particular, we have shown that $H^{loc}_{x,y}$ is absolutely continuous with respect to $m_x$ and $m_y$ with uniformly bounded Jacobian. Since $x,y$ are arbitrary, it follows that $\cF^{cs}$ is absolutely continuous. 
	
	Note that all the previous arguments apply to $f^{-1}$. We get that $\cF^{cu}$ is also absolutely continuous, and therefore $\cF^c$ is absolutely continuous. It then follows from the main result of \cite{AVW} that $f$ is the time-one map of a smooth flow.
	
	{
	Next we consider Assumption (2). In this case we have $\lambda^s(\mu_{{\operatorname{MME}}}) = \lambda^u(\mu_{{\operatorname{MME}}}) = h_{top}(f)$, so $\mu_{\operatorname{MME}}$ is also a $u$-MME due to Theorem \ref{t.geodesic}. We also get that $\mu_{\operatorname{MME}}$ has zero center exponent, and is a Gibbs $u$-state by \cite{LY1}. As before, Theorem \ref{t.2} shows that the log Jacobian of $f$ is cohomologous to a constant via a function $\phi$ which is $\mu_{\operatorname{MME}}$-measurable.
	
	In addition, the support of $\mu_{\operatorname{MME}}$ consists of entire $u$-leaves, by Theorem \ref{t.properties}. By considering $f^{-1}$ we get that $\supp(\mu_{\operatorname{MME}})$ also consists of entire $s$-leaves. Since $f$ is accessible, we have $\supp(\mu_{\operatorname{MME}}) = \bM$. 
	Then, it is straightforward to check that the previous proof of $\cF^{cs}$ being absolutely continuous still applies in this case. By considering $f^{-1}$ we also obtain that $\cF^{cu}$ is also absolutely continuous, and therefore $f$ is the time-one map of a smooth flow $X$ due to the main result of  \cite{AVW}. In particular, $\mu_{\operatorname{MME}} = m$ since Gibbs $u$-state of $X$ is unique, and the  solution $\phi$ can be made continuous due to \cite[Theorem A]{Wilkinson}.}
	 
\end{proof}

Note that under both Assumptions (1) and (2) of Theorem \ref{t.geodesic.rigidity}, $m$ coincides with the unique measure of maximal entropy of a smooth vector field $X$. Using the main result of \cite{DLVY} we obtain:
\begin{corollary}
	Under the assumptions of Theorem \ref{t.geodesic.rigidity}, $f$ is smoothly conjugated to the time-one map of an algebraic flow.
\end{corollary}

\section*{Acknowledgment}

\end{document}